\documentclass[11pt]{article}
\pdfoutput=1

\usepackage[margin=1.3in]{geometry}

\usepackage{graphicx}
\usepackage{amssymb,amsmath}
\usepackage{amsthm,enumerate}
\usepackage{hyperref,xcolor}
\usepackage{mathrsfs}
\usepackage{extarrows}
\usepackage{textcomp}

\makeatletter

\newdimen\bibspace
\makeatother

\numberwithin{equation}{section}

\newtheorem{theorem}{Theorem}[section]
\newtheorem{lemma}[theorem]{Lemma}
\newtheorem{proposition}[theorem]{Proposition}

\newtheorem{corollary}[theorem]{Corollary}
\newtheorem{remark}[theorem]{Remark}

\def\<{\langle}
\def\>{\rangle}

\def\XXint#1#2#3{{\setbox0=\hbox{$#1{#2#3}{\int}$ }
\vcenter{\hbox{$#2#3$ }}\kern-.6\wd0}}

\begin{document}

\title{Asymptotic expansion  at infinity of solutions of\\ special Lagrangian  equations}

\author{Zixiao Liu,\quad Jiguang Bao\footnote{Supported in part by Natural Science Foundation of China (11871102 and 11631002).}}
\date{\today}

\maketitle

\begin{abstract}
We obtain a quantitative high order expansion at infinity of solutions for a family of fully nonlinear elliptic equations on exterior domain, refine the study of the asymptotic behavior of the Monge-Amp\`ere equation, the special Lagrangian equation and other elliptic equations, and give the precise gap between exterior maximal (or minimal) gradient graph and the entire case.


 {\textbf{MSC~2020:}}~~ 35J60;~~35B40
\end{abstract}

\section{Introduction}

Consider a Lagrangian submanifold of $\mathbb{R}^{n} \times \mathbb{R}^{n}$   that can be represented locally as a gradient graph $(x, D u(x))$.
In 2010, Warren \cite{Warren} first studied the minimal or maximal Lagrangian graph in $\left(\mathbb{R}^{n} \times \mathbb{R}^{n}, g_{\tau}\right)$, where
\begin{equation*}
g_{\tau}=\sin \tau \delta_{0}+\cos \tau g_{0},\quad\tau \in\left[0, \frac{\pi}{2}\right],
\end{equation*}
is the linearly combined metric of standard Euclidean metric
\begin{equation*}
\delta_{0}=\sum_{i=1}^{n} d x_{i} \otimes d x_{i}+\sum_{j=1}^{n} d y_{j} \otimes d y_{j},
\end{equation*}
and the pseudo-Euclidean metric
\begin{equation*}
g_{0}=\sum_{i=1}^{n} d x_{i} \otimes d y_{i}+ \sum_{j=1}^{n} d y_{j} \otimes d x_{j}.
\end{equation*}
He proved that if $u\in C^2(\Omega)$
is a solution of
\begin{equation}\label{equation}
F_{\tau}\left(\lambda\left(D^{2} u\right)\right)=C_0,\quad x\in\Omega,
\end{equation}
where $\Omega\subset\mathbb R^n$ is a domain,
then the volume of
$(x,Du(x))$ is a maximal for $\tau\in (0,\frac{\pi}{4})$ and minimal for $\tau\in (\frac{\pi}{4},\frac{\pi}{2})$ among all homologous, $C^1$, space-like $n$-surfaces in $\left(\mathbb{R}^{n} \times \mathbb{R}^{n}, g_{\tau}\right)$.
In (\ref{equation}), $C_0$ is a constant,  $\lambda(D^2u)=(\lambda_1,\lambda_2,\cdots,\lambda_n)$ are $n$ eigenvalues of  Hessian matrix $D^2u$ and  $$
F_{\tau}(\lambda):=\left\{
\begin{array}{ccc}
\displaystyle  \frac{1}{n} \sum_{i=1}^{n} \ln \lambda_{i}, & \tau=0,\\
\displaystyle  \frac{\sqrt{a^{2}+1}}{2 b} \sum_{i=1}^{n} \ln \frac{\lambda_{i}+a-b}{\lambda_{i}+a+b},
  & 0<\tau<\frac{\pi}{4},\\
  \displaystyle-\sqrt{2} \sum_{i=1}^{n} \frac{1}{1+\lambda_{i}}, & \tau=\frac{\pi}{4},\\
  \displaystyle\frac{\sqrt{a^{2}+1}}{b} \sum_{i=1}^{n} \arctan \displaystyle\frac{\lambda_{i}+a-b}{\lambda_{i}+a+b}, &
  \frac{\pi}{4}<\tau<\frac{\pi}{2},\\
  \displaystyle\sum_{i=1}^{n} \arctan \lambda_{i}, & \tau=\frac{\pi}{2},\\
\end{array}
\right.
$$
$a=\cot \tau, b=\sqrt{\left|\cot ^{2} \tau-1\right|}$.

If $\tau=0$, then (\ref{equation}) becomes the Monge-Amp\`ere equation $$
\det D^2u=e^{nC_0}.
$$
The classical theorem by J\"orgens \cite{Jorgens}, Calabi \cite{Calabi} and Pogorelov \cite{Pogorelov} states that any convex classical solution of $\operatorname{det}D^{2} u=1$ on $\mathbb{R}^{n}$ must be a quadratic polynomial. See Cheng-Yau \cite{ChengandYau}, Caffarelli \cite{7} and Jost-Xin \cite{JostandXin} for different proofs and extensions.
For the Monge-Amp\`ere equation in exterior domain, there are exterior J\"orgens-Calabi-Pogorelov type results by Ferrer-Mart\'{\i}nez-Mil\'{a}n \cite{FMM99} for $n=2$ and Caffarelli-Li \cite{CL}, which state that any convex solution must be asymptotic to quadratic polynomials (for $n=2$ we need additional $\ln$-term) near infinity.
There is finer result for $n\geq 3$  by Hong \cite{RemarkMA-2020}, which states that
$$u=\frac{1}{2} x^TA x+\beta \cdot x+\gamma+d( x^T A x)^{\frac{2-n}{2}}+O(|x|^{1-n})$$
as $|x|\rightarrow+\infty$
for some  $\beta\in\mathbb R^n, \gamma, d\in\mathbb R$, $A\in \mathtt{Sym}(n)$ and $\det A=1$, where
$x^T$ denote the transpose of  vector $x\in\mathbb R^n$ and $\mathtt{Sym}(n)$ denote the set of symmetric $n\times n$ matrix,

If $\tau=\frac{\pi}{2}$, then (\ref{equation}) becomes the special Lagrangian equation
\begin{equation}\label{equ-spl}
\sum_{i=1}^{n} \arctan \lambda_{i}\left(D^{2} u\right)=C_0.
\end{equation}
There are Bernstein-type results by Yuan \cite{Yu.Yuan1,Yu.Yuan2}, which state that any classical solution of \eqref{equ-spl} on $\mathbb R^n$ with either
\begin{equation}\label{equ-cond-spl}
   D^2u\geq \left\{
  \begin{array}{lll}
    -KI, & n\leq 4,\\
    -(\frac{1}{\sqrt 3}+\epsilon(n))I, & n\geq 5,\\
  \end{array}
  \right.
\end{equation}
or
\begin{equation}\label{equ-cond-spl2}
  |C_0|>\frac{n-2}{2}\pi
\end{equation}
must be a quadratic polynomial, where  $I$ denote the unit $n\times n$ matrix, $K$ is an arbitrary large constant and $\epsilon(n)$ is a small dimensional constant. For special Lagrangian equations in exterior domain, there is exterior Bernstein-type result by Li-Li-Yuan \cite{ExteriorLiouville}, which states that any classical solution of \eqref{equ-spl} on exterior domain with \eqref{equ-cond-spl} or \eqref{equ-cond-spl2} must be asymptotic to quadratic polynomial  (for $n=2$ we need additional $\ln$-term) near infinity.

If $\tau=\frac{\pi}{4}$, then \eqref{equation} is a translated inverse harmonic Hessian equation $\sum_{i=1}^n\frac{1}{\lambda_i(D^2u)}=1$, which is a special form of Hessian quotient equation. There is Bernstein-type result for Hessian quotient equations by Bao-Chen-Guan-Ji \cite{BCGJ}.

For general $\tau\in [0,\frac{\pi}{2}]$,
Warren \cite{Warren} proved the Bernstein-type results under suitable semi-convex conditions by the results of J\"orgens \cite{Jorgens}-Calabi \cite{Calabi}-Pogorelov \cite{Pogorelov}, Flanders \cite{Flanders} and Yuan \cite{Yu.Yuan1,Yu.Yuan2}.

There are also many study on asymptotic behavior and asymptotic expansions of other geometric curvature equations. Most recently, Han-Li-Li \cite{Han2019-Expansion} proved the asymptotic expansion near isolated singular point of the Yamabe equation and $\sigma_k$-Yamabe equation, refined the previous study by Caffarelli-Gidas-Spruck \cite{CGS}, Korevaar-Mazzeo-Pacard-Schoen \cite{KMPS},  Han-Li-Teixeira \cite{Han-Li-T-Simgak} etc.
The expansion near isolated singularity is related to expansion at infinity under Kelvin transform. Many other study on the asymptotics of related elliptic equations on punctured domain, exterior domain or half cylinders can be found in \cite{20,CCLin-Duke95,CCLin-CPAM97,Li-Zhang-Harnack,19,Marques-CVPDE,Li-FracHardy,
xiong-2020,
chen-liu-cylinder2020} etc.

In this paper, we obtain asymptotic expansions at infinity of classical solutions of
\begin{equation}\label{Equ-exterior-domain}
F_{\tau}(\lambda(D^2u))=C_0\quad\text{in}~\mathbb R^n\setminus\overline{B_1},
\end{equation}
where $n\geq 3$ and $B_1$ denote the unit ball centered at origin in $\mathbb R^n$.
This refines previous study including \cite{CL,ExteriorLiouville,RemarkMA-2020} etc.

Our first result focus on the asymptotic expansion at infinity of radially symmetric classical solutions of \eqref{Equ-exterior-domain}.
\begin{theorem}\label{Thm-main-radial}
  Let $u \in C^{2}\left(\mathbb{R}^{n} \backslash \overline{B_1}\right)$ be a classical radially symmetric solution of
\eqref{Equ-exterior-domain}.
Suppose either of the following holds
\begin{enumerate}[(i)]
\item \label{Case-MA} $D^2u>0$ for $\tau=0$;
\item \label{Case-small} $D^2u>(-a+b)I$ for $\tau \in (0,\frac{\pi}{4})$;

\item \label{Case-inverse} $D^2u>-I$ for $\tau=\frac{\pi}{4}$;

\item \label{Case-large} either
\begin{equation}\label{condition-temp-2}
D^2u>-(a+b)I\quad\text{and}\quad
    D^2u\geq
    \left\{
    \begin{array}{lll}
      -(a+bK)I, & n\leq 4,\\
      -(a+\frac{b}{\sqrt 3}+b\epsilon(n))I, & n\geq 5,\\
    \end{array}
    \right.
    \end{equation}
    where $K,\epsilon(n)$ are as in \eqref{equ-cond-spl}
     or
   \begin{equation}\label{condition-temp-3}
    D^2u> -(a+b)I\quad\text{and}\quad \left|\frac{bC_0}{\sqrt{a^2+1}}+\frac{n\pi}{4}\right|>\frac{n-2}{2}\pi
    \end{equation}
for $\tau\in(\frac{\pi}{4},\frac{\pi}{2})$;

\item \label{Case-SPL} either \eqref{equ-cond-spl} or \eqref{equ-cond-spl2} for $\tau=\frac{\pi}{2}$.
\end{enumerate}
Then $u$ is analytic at infinity and there exist constants $c_2,c_{-k}$ with $k=1,2,\cdots$, such that
\begin{equation}\label{Equ-expansion-radial}
  u(x)=c_2|x|^2+c_0+|x|^2\sum_{k=1}^{+\infty}c_{-k}(c|x|^{-n})^k
\end{equation}
for sufficiently large $|x|$,
where $c_0,c$ are arbitrary constants.
\end{theorem}

\begin{remark}
  Actually, in the process of proving Theorem \ref{Thm-main-radial}, we also study the existence of all symmetric solutions of \eqref{Equ-exterior-domain}.
\end{remark}

Our second result considers higher order expansions of general classical solution of \eqref{Equ-exterior-domain}, which gives the precise gap between exterior maximal (or minimal) gradient graph and the entire case. Hereinafter, we let $\varphi=O_l(|x|^{-k_1}(\ln|x|)^{k_2})$ with $l\in\mathbb N, k_1,k_2\geq 0$ denote $$|D^k\varphi|=O(|x|^{-k_1-k}(\ln|x|)^{k_2})
\quad\text{as}~|x|\rightarrow+\infty
$$ for all $0\leq k\leq l$. Let
$\mathcal H_k^n$ denote the $k$-order spherical harmonic function space in dimension $n$ and
$DF_{\tau}(\lambda(A))$ denotes the matrix with elements being value of partial derivative of $F_{\tau}(\lambda(M))$ w.r.t $M_{ij}$ variable at matrix $A$.

\begin{theorem}\label{Thm-main-2}
Let $u \in C^{2}\left(\mathbb{R}^{n} \backslash \overline{B_1}\right)$ be a classical solution of
\eqref{Equ-exterior-domain}.
Suppose either of \eqref{Case-MA}-\eqref{Case-SPL} holds.
Then there exist $\gamma\in\mathbb R, \beta\in\mathbb R^n, A\in\mathtt{Sym}(n)$, $F_{\tau}(\lambda(A))=C_0$ and $c_{k}(\theta)\in\mathcal H_k^n$ with $k=0,1,\cdots,n-1$  such that
  \begin{equation}
  \label{ConvergenceSpeed2}
  \begin{array}{llll}
  &\displaystyle u ( x ) - \left( \frac { 1 } { 2 } x ^ TA x + \beta \cdot x + \gamma \right)\\
  -&\displaystyle\left(
  c_0(x^T(DF_{\tau}(\lambda(A)))^{-1}x)^{\frac{2-n}{2}}
  +\sum_{k=1}^{n-1}c_k(\theta)\left(x^T(DF_{\tau}(\lambda(A)))^{-1} x\right)^{\frac{2-n-k}{2}}
  \right)\\=&O_l(|x|^{2-2n}\ln |x|)
  \end{array}
  \end{equation}
  as $|x|\rightarrow+\infty$   for all $l\in\mathbb N$, where
  \begin{equation*}
\theta=\frac{(DF_{\tau}(\lambda(A)))^{-\frac{1}{2}}x}{\left(x^T(
DF_{\tau}(\lambda(A)))^{-1}x\right)^{\frac{1}{2}}}.
\end{equation*}
\end{theorem}

\begin{remark}
The matrix $A$ in Theorem \ref{Thm-main-2} also satisfies $A\geq 0$ in case \eqref{Case-MA}, $A\geq (-a+b)I$ in case \eqref{Case-small}, $A\geq -I$ in case \eqref{Case-inverse}, $A\geq -(a+b)I$ with \eqref{condition-temp-2} or $A>-\infty$ in case \eqref{Case-large} and \eqref{equ-cond-spl} or $A>-\infty$ in case \eqref{Case-SPL} respectively.

By comparison principle as in \cite{CL}, the global Bernstein-type results \cite{Warren} follow from these exterior behavior results in Theorem \ref{Thm-main-2}.
\end{remark}
\begin{remark}

Theorem \ref{Thm-main-2} can be extended to a more general result on classical solution of
\begin{equation}\label{Equ-fullyNonlinear-intro}
  F(D^2u)=C_0\quad\text{in}~\mathbb R^n\setminus\overline{B_1}
\end{equation}
with
\begin{equation}
  \label{ConvergenceSpeed1}
  u ( x ) - \left( \frac { 1 } { 2 } x ^ T A x + \beta \cdot x + \gamma \right)=O_l(|x|^{2-n})
  \end{equation}
for all $l\in\mathbb N$, where $F$ is smooth, $\gamma\in\mathbb R$, $\beta\in\mathbb R^n$, $A\in\mathtt{Sym}(n)$, $F(A)=C_0$ and $DF(A)>0$ . Then the same result in Theorem \ref{Thm-main-2} holds with $DF_{\tau}(\lambda (A))$ replaced by $DF(A)$. See Lemma \ref{lem-expansion-second} for more details.

By
Theorem 2.1 in \cite{ExteriorLiouville}, if $u$ is a classical solution of \eqref{Equ-fullyNonlinear-intro} with bounded Hessian matrix, where $F$ is smooth, uniformly elliptic and $\{M|F(M)=C_0\}$ is convex, then \eqref{ConvergenceSpeed1} holds.
\end{remark}
\begin{remark}
  Expansion \eqref{ConvergenceSpeed2} is optimal in the sense that the series of $k$ in \eqref{ConvergenceSpeed2} cannot be taken up to $n$ since $c_n(\theta)\not\in \mathcal H_n^n$ in general. See for instance \eqref{Equ-expansion-radial} in Theorem \ref{Thm-main-radial}.
\end{remark}

The paper is organized as follows. In section \ref{Sec-radial-Solution} we
classify radially symmetric classical solutions of \eqref{Equ-exterior-domain}
and  prove Theorem \ref{Thm-main-radial}. In section \ref{Sec-Expansion-General} we prove Theorem \ref{Thm-main-2} by finer analysis on the linearized equation, which includes the existence of ``fast decay'' solution of Poisson equations and spherical harmonic expansions of harmonic functions.

\section{Asymptotic expansions of radially symmetric solutions}\label{Sec-radial-Solution}

In this section, we calculate the asymptotic behavior of all radially symmetric classical solution of \eqref{Equ-exterior-domain} by solving the ODEs.
Since $u(x)=:U(|x|)$ is radially symmetric, the $n$ eigenvalues of $D^2u(x)$ are exactly
  \begin{equation}\label{Equ-eigenvlaues}
  \lambda_1=U''(r),~\lambda_2,\cdots,\lambda_n=\dfrac{U'(r)}{r},
  \end{equation}
  where $r=|x|>1$.

In the following, we divide this section into 5 subsections according to 5 cases.

\subsection{$\tau=0$ Case}

When $\tau=0$, equation \eqref{Equ-exterior-domain} reads
\begin{equation}\label{equ-equ-0}
\dfrac{1}{n}\sum_{i=1}^n\ln\lambda_i=C_0,\quad|x|>1.
\end{equation}
By \eqref{Equ-eigenvlaues}, \eqref{equ-equ-0} becomes
$$
U''(r)\cdot (\frac{U'(r)}{r})^{n-1}=C'\quad\text{in}~r>1,
$$
where $C':=\exp(nC_0)\in (0,+\infty)$. Let
\begin{equation}\label{equ-W}
W(r):=\frac{U'}{r}.
\end{equation}
In order to make \eqref{equ-equ-0} well-defined, $W(r)\in (0,+\infty)$ for all $r>1$. By a direct computation,
$$
rW'\cdot W^{n-1}+W^{n}=C'.
$$
This is a separable differential equation, which leads to
\begin{equation}\label{Temp-0}
W^n-C'=cr^{-n}
\end{equation}
for some constant $c$, for all $r>1$. Thus
$$
W(r)=\left(cr^{-n}+C'\right)^{\frac{1}{n}}\quad\forall~r>1.
$$
As long as $c\geq -C'$, $W(r)>0$ exists for all $r>1$ and implies
\begin{equation}\label{Equ-def-u-0}
u(x)=\int_1^{|x|}\tau(c\tau^{-n}+C')^{\frac{1}{n}}\mathtt{d}\tau+c'_0
\end{equation}
for $c_0'\in\mathbb R$. Furthermore,
\begin{equation}\label{Equ-asy-radial-0}
u(x)=\frac{1}{2}C'^{\frac{1}{n}}|x|^2+c_0+
C'^{\frac{1}{n}}|x|^2\sum_{j=1}^{+\infty}\frac{(1/n)\cdots(1/n-j+1)}{(2-nj)j!}(\frac{c|x|^{-n}}{C'})^j
\end{equation}
for $|x|>\max\{1,(\frac{|c|}{C'})^{\frac{1}{n}}\}$ and any $c_0\in\mathbb R$.

Thus in this case, we have
\begin{theorem}
  Let $u\in C^2(\mathbb R^n\setminus\overline{B_1})$ be a radially symmetric solution of \eqref{equ-equ-0}, then $u$ is given by \eqref{Equ-def-u-0}, where $c_0\in\mathbb R$ and $c\geq -e^{nC_0}$. Moreover, $u$ is analytic at infinity with expansion \eqref{Equ-asy-radial-0}.
\end{theorem}

\subsection{$\tau\in (0,\frac{\pi}{4})$ Case}

When $\tau\in (0,\frac{\pi}{4})$, equation \eqref{Equ-exterior-domain} reads
\begin{equation}\label{Equ-equ-1}
\frac{\sqrt{a^{2}+1}}{2 b} \sum_{i=1}^{n} \ln \frac{\lambda_{i}+a-b}{\lambda_{i}+a+b}=C_0,\quad|x|>1.
\end{equation}
We may assume without loss of generality that $C_0\leq 0$, otherwise we consider  $\bar u(x):=-u(x)-a|x|^2$ instead, which satisfies $\bar\lambda_i=-(\lambda_i+2a)$ and hence
$$
\frac{\sqrt{a^{2}+1}}{2 b} \sum_{i=1}^{n} \ln \frac{\bar \lambda_i+a-b}{\bar \lambda_i+a+b}=-C_0.
$$
By \eqref{Equ-eigenvlaues}, \eqref{Equ-equ-1} becomes
$$
\frac{U''(r)+a-b}{U''(r)+a+b}\cdot\left(
\frac{\frac{U'(r)}{r}+a-b}{\frac{U'(r)}{r}+a+b}
\right)^{n-1}=C'\quad\text{in}~r>1,
$$
where $C^{\prime}:=\exp \left(\frac{2 b}{\sqrt{a^{2}+1}} C_{0}\right)\in (0,1]$.
Let
\begin{equation*}
W(r):=\frac{1}{2 b}(\frac{U'(r)}{r}+a+b).
\end{equation*}
In order to make \eqref{Equ-equ-1} well-defined, $W(r)\in (-\infty,0)\cup (1,+\infty)$ for all $r>1.$
By a direct computation,
\begin{equation}\label{equ-welldefine-1}
rW'\cdot\left(C'W^{n-1}-
(W-1)^{n-1}
\right)+\left(
C'W^n-(W-1)^n
\right)=0.
\end{equation}
This is a separable differential equation, which leads to
\begin{equation}\label{Temp-1}
C'W^n(r)-(W(r)-1)^n=cr^{-n},
\end{equation}
for some constant $c$, for all $r>1$.

If $C'=1$, then $|W^n-(W-1)^n|>1$ for $W\in (-\infty,0)\cup (1,+\infty)$, which implies \eqref{Temp-1} have no solution on entire $r>1$.

If $C'\in(0,1)$ and $c=0$, then \eqref{Temp-1} admits a constant solution $W(r)\equiv \frac{1}{1-\sqrt[n]{C^{\prime}}}$, which implies  quadratic solutions $u(x)=
-\left(\frac{a+b}{2}-\frac{b}{1-\sqrt[n]{C^{\prime}}})\right)|x|^{2}+c_0
$ for all $c_0\in\mathbb R$.

If $C'\in(0,1)$ and $c\not=0$, since \eqref{Temp-1} holds for all $r>1$, we consider the inverse function $w(\xi)$ of $\xi=G(w)$ on entire $(0,c)$ or $(c,0)$, where
\begin{equation}\label{Equ-alg-1}
\left\{
\begin{array}{llll}
  G(w):=C'w^n-(w-1)^n\\
  w\in (-\infty,0)\cup(1,+\infty).\\
\end{array}
\right.
\end{equation}
See for example the following two pictures of $G(w)$.
\begin{figure}[htbp]
  \centering
  \includegraphics[width=0.45\linewidth]{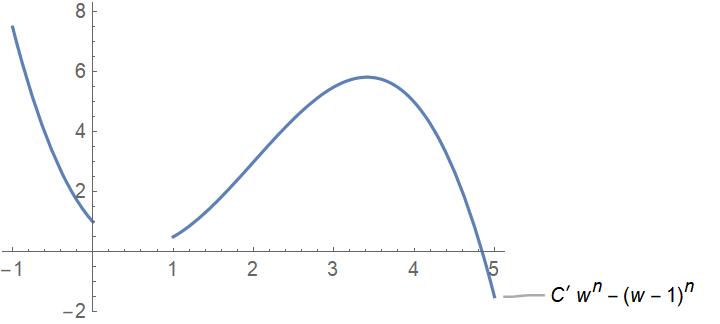}
  \includegraphics[width=0.45\linewidth]{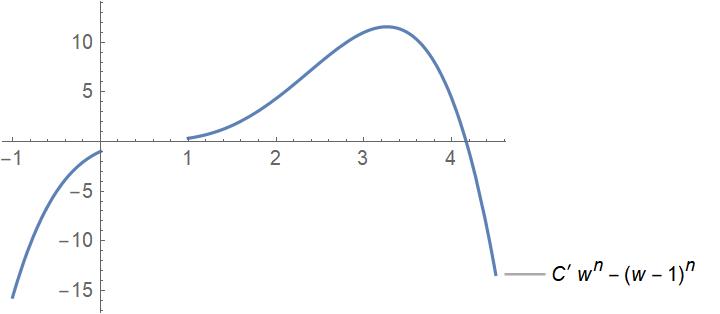}
\end{figure}

When $n$ is odd,
$$
G'(w)=n\left(C'w^{n-1}-(w-1)^{n-1}\right)\left\{
\begin{array}{lll}
  <0 & w\in (-\infty,0)\cup
  (\frac{1}{1-\sqrt[n-1]{C'}},+\infty), \\
  >0 & w\in (1,\frac{1}{1-\sqrt[n-1]{C'}}).\\
\end{array}
\right.
$$
At the endpoints, we have $G(0)=1$ with
\begin{equation}\label{equ-temp-6}
  G(1)=C'>0,\quad\text{and}\quad G\left(\frac{1}{1-\sqrt[n-1]{C^{\prime}}}\right)=
  \frac{C'}{(1-\sqrt[n-1]{C'})^{n-1}}>0,
\end{equation}
Furthermore,
  $$
  G\left(\frac{1}{1-\sqrt[n]{C^{\prime}}}\right)=0,\quad G(-\infty)=+\infty,\quad \text{and}\quad G(+\infty)=-\infty.
  $$

When $n$ is even,
$$
G'(w)\left\{
\begin{array}{lll}
  <0 & w\in
  (\frac{1}{1-\sqrt[n-1]{C'}},+\infty), \\
  >0 & w\in (-\infty,0)\cup(1,\frac{1}{1-\sqrt[n-1]{C'}}),\\
\end{array}
\right.
$$
At the endpoints, we have $G(0)=-1$ with \eqref{equ-temp-6}. Furthermore,
$$
G\left(\frac{1}{1-\sqrt[n]{C^{\prime}}}\right)
=0, \quad G(-\infty)=-\infty, \quad G(+\infty)=-\infty.
$$

Hence in the neighbourhood of origin, we have a unique inverse function $w(\xi)$ of $\xi=G(w)$ in $w\in (\frac{1}{1-\sqrt[n-1]{C'}},+\infty)$
and $\xi\in (-\infty, \frac{C^{\prime}}{\left(1-\sqrt[n-1]{C^{\prime}}\right)^{n-1}}]$.
  By \eqref{Temp-1} and the discussion above,
  \begin{equation}\label{Temp-4}
  W(r)=w(cr^{-n})\quad\forall~r>1.
  \end{equation}
  As long as $c\leq \frac{C^{\prime}}{\left(1-\sqrt[n-1]{C^{\prime}}\right)^{n-1}}$,
  $w(cr^{-n})$ exists for all $r>1$ and implies
\begin{equation}\label{Equ-def-u-1}
u(x)=-\frac{(a+b-2 b w(0))}{2}|x|^{2}+2b\int_{+\infty}^{|x|}(w(c\tau^{-n})-w(0))\cdot \tau \mathrm{d} \tau+c_{0}
\end{equation}
for $c_0\in\mathbb R$.

 Especially since $G'(\frac{1}{1-\sqrt[n]{C^{\prime}}})\not=0$, $w(\xi)$ is analytic in a neighbourhood of $\xi=0$. Hence
\begin{equation}\label{Equ-asy-radial-1}u(x)=-\frac{1}{2}(a+b-2 b w(0))|x|^{2}+c_{0}-2b|x|^{2} \sum_{j=1}^{\infty} \frac{w^{(j)}(0)}{(n j-2) j !}\left(|x|^{-n}c\right)^{j}
\end{equation}
 for sufficiently large $|x|$.

 Thus in this case, we have
\begin{theorem}
  Let $u \in C^{2}\left(\mathbb{R}^{n} \backslash \overline{B_{1}}\right)$ be a radially symmetric solution of \eqref{Equ-equ-1} with $C_0\neq 0$, then $u$ is given by
  $$
  u(x)=\frac{-a+2bw(0)+b\cdot \frac{C_0}{|C_0|}}{2}|x|^{2}+2 b\cdot\frac{C_0}{|C_0|}\int^{+\infty}_{|x|}\left(w\left(c \tau^{-n}\right)-w(0)\right) \cdot \tau \mathrm{d} \tau+c_{0},
  $$
  where  $c_0\in\mathbb{R}$, $w(\xi)$ is the inverse function of
  $$
  \exp\left(\dfrac{-2b}{\sqrt{a^2+1}}|C_0|\right)w^n-(w-1)^n=\xi\quad\text{with}\quad
  w(0)=\dfrac{1}{1-\exp\left(\frac{-2b}{n\sqrt{a^2+1}}|C_0|\right)},
  $$
  and
  $$c\leq \frac{1}{(\exp(\frac{2b}{(n-1)\sqrt{a^2+1}}|C_0|)-1)^{n-1}}.$$
 Moreover, $u$ is analytic at infinity with expansion \eqref{Equ-asy-radial-1}.

If $C_0=0$, there are no radially symmetric classical solution of \eqref{Equ-equ-1} on exterior domain.
\end{theorem}

Under  condition \eqref{Case-small},
$$
\ln \frac{\lambda_{i}+a-b}{\lambda_{i}+a+b}<0\quad\forall~i=1,2,\cdots,n,
$$
hence \eqref{Equ-equ-1} implies $C_0<0$ and then radially symmetric classical solutions are given by \eqref{Equ-def-u-1}.

\subsection{$\tau=\frac{\pi}{4}$ Case}

When $\tau=\frac{\pi}{4}$, equation \eqref{Equ-exterior-domain} reads
\begin{equation}\label{equ-equ-2}
-\sqrt{2} \sum_{i=1}^{n} \frac{1}{1+\lambda_{i}}=C_0,\quad|x|>1.
\end{equation}
Let
$$
W(r):=\dfrac{U'(r)}{r}+1.
$$
In order to make \eqref{equ-equ-2} well-defined, $W(r)\in (-\infty,0)\cup(0,+\infty)$ for all $r>1$.
By a direct computation,
\begin{equation}\label{equ-welldefine-2}
(n-1-C'W)rW'+nW-C'W^2=0,
\end{equation}
where $C':=-\frac{C_0}{\sqrt 2}\in\mathbb R$.

When $C'=0$,  \eqref{equ-welldefine-2} leads to
$$
W(r)=cr^{-\frac{n}{n-1}},
$$
for some constant $c$, for all $r>1$. As long as $c\not=0$, $W(r)\not=0$ for all $r>1.$ Thus in this case,
\begin{equation}\label{Equ-asy-radial-4}
u(x)=-\frac{1}{2}|x|^2+\dfrac{n-1}{n-2}c|x|^{-\frac{n-2}{n-1}}+c_0.
\end{equation}

When $C'\not=0$, we may assume without loss of generality that $C'=1$, otherwise we consider
$
W(r):=\frac{1}{C'}(\frac{U'(r)}{r}+1)
$ instead. In this case, \eqref{equ-welldefine-2} is a separable differential equation that leads to
\begin{equation}\label{Temp-5}
W^n(r)-nW^{n-1}(r)=cr^{-n},
\end{equation}
for some constant $c$, for all $r>1$.

If $c=0$, then  \eqref{Temp-5} admits a constant solution $W(r)\equiv n$, which implies quadratic solutions $u(x)=-\frac{1}{2}\left(1-nC'\right)|x|^{2}+c_0$ for all $c_0\in\mathbb R$.

If $c\not=0$, since \eqref{Temp-5} holds for all $r>1$, we consider the inverse function $w(\xi)$ of $\xi=G(w)$ on entire $(0,c)$ or $(c,0)$, where
\begin{equation}\label{Equ-alg-2}
\left\{
\begin{array}{lll}
  G(w):=w^n-nw^{n-1}\\
  w\in (-\infty,0)\cup(0,+\infty).\\
\end{array}
\right.
\end{equation}
See for example the following two pictures of $G(w)$.
\begin{figure}[htbp]
  \centering
  \includegraphics[width=0.4\linewidth]{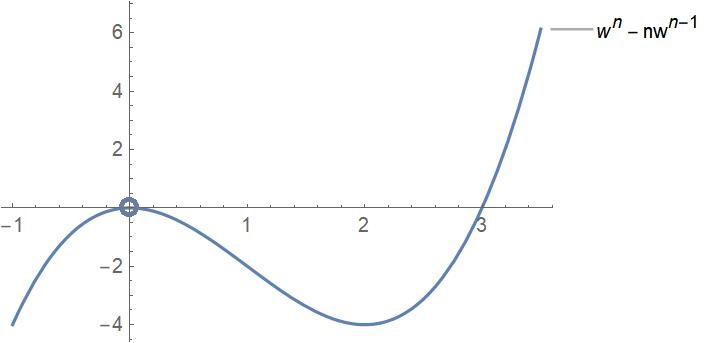}
  \includegraphics[width=0.4\linewidth]{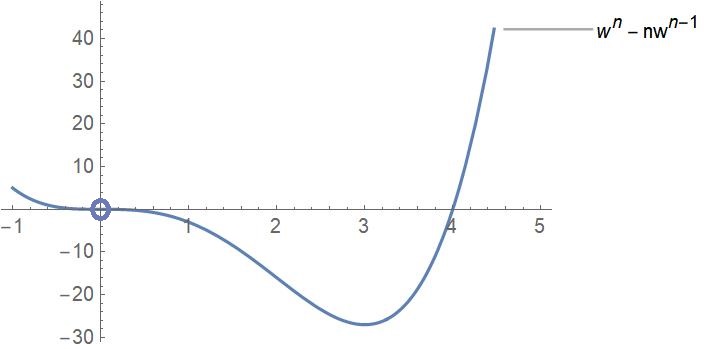}
\end{figure}

When $n$ is odd,
$$
G'(w)=nw^{n-1}-n(n-1)w^{n-2}\left\{
\begin{array}{llll}
  >0 & w\in (-\infty,0)\cup (n-1,+\infty),\\
  <0 & w\in (0,n-1).
\end{array}
\right.
$$
At the end points,
  \begin{equation}\label{Equ-criticalpt-2}
  G(0)=0,~G(n-1)=-(n-1)^{n-1}<0.
  \end{equation}
  Furthermore,
  $$
  G(n)=0,~G(-\infty)=-\infty,~\text{and}~G(+\infty)=+\infty.
  $$
  Since there are three monotone domain with range being a half-neighbourhood of origin, we have three inverse functions  $w_1(\xi),w_2(\xi),w_3(\xi)$ of $\xi=G(w)$. They exist in $\xi\in (-(n-1)^{n-1},+\infty)$, $\xi\in (-(n-1)^{n-1},0)$,  $\xi\in (-\infty,0)$ with $w_1(\xi)\in (n-1,+\infty)$, $w_2(\xi)\in (0,n-1)$,
  $w_3(\xi)\in (-\infty,0)$ respectively.

When $n$ is even,
$$
G'(w)\left\{
\begin{array}{lll}
  >0 & w\in (n-1,+\infty),\\
  <0 & w\in (-\infty,0)\cup(0,n-1).\\
\end{array}
\right.
$$
At the end points, we still have \eqref{Equ-criticalpt-2}. Furthermore,
  $$
  G(n)=0,\quad G(-\infty)=+\infty\quad\text {and}\quad G(+\infty)=+\infty.
  $$
Similarly, we have three inverse functions
 $w_1(\xi),w_2(\xi),w_3(\xi)$ exist in $\xi\in (-(n-1)^{n-1},+\infty)$, $\xi\in (-(n-1)^{n-1},0)$,  $\xi\in (0,+\infty)$ with $w_1(\xi)\in (n-1,+\infty)$, $w_2(\xi)\in (0,n-1)$,
 $w_3(\xi)\in (-\infty,0)$ respectively.

By \eqref{Temp-5} and the discussion above,
$$
W(r)=w_p(cr^{-n})\quad\forall~r>1,
$$
for some $p\in\{1,2,3\}$. When $c\geq -(n-1)^{n-1}$ or $0>c\geq -(n-1)^{n-1}$ or $c>0$,  $w_p(cr^{-n})$ exists for all $r>1$ with $p\in\{1,2,3\}$ respectively and implies
\begin{equation}\label{Equ-def-u-2}
u(x)=-\dfrac{1}{2}\left(1-C'w_p(0)\right)|x|^2-C'\int_{+\infty}^{|x|}\left(-w_p(c\tau^{-n})+w_p(0)
\right)\cdot \tau\mathtt{d}\tau+c_0
\end{equation}
for $c_0\in\mathbb R.$

Especially for $p=1$, $G'(n)>0$ and hence $w_1(\xi)$ is analytic in an neighbourhood of origin. Hence
\begin{equation}\label{Equ-asy-radial-2}
u(x)=-\dfrac{1}{2}\left(1-C'w_1(0)\right)|x|^2+c_0-C'|x|^2\sum_{j=1}^{\infty}\frac{w_1^{(j)}(0)}{(n j-2) j !}\left(|x|^{-n}c\right)^{j}
\end{equation}
for sufficiently large $|x|$. For $p=2,3$, we prove that the inverse functions $w_p(\xi)$ are not analytic in a neighbourhood of $\xi=0$. By contradiction, suppose $w(\xi)=\sum_{j=j_0}^{+\infty}c_j\xi^j$ in a neighbourhood of origin with $c_{j_0}\not=0$ and $j_0\geq 1$. Then
$$
G(w(\xi))=\left(\sum_{j=j_0}^{+\infty}c_j\xi^j\right)^n-n\left(\sum_{j=j_0}^{+\infty}c_j\xi^j\right)^{n-1}
=-nc_{j_0}^{n-1}\xi^{j_0(n-1)}+o(\xi^{j_0(n-1)})=O(\xi^{n-1}),
$$
contradicting to $G(w(\xi))=\xi$.

Thus in this case, we have
\begin{theorem}
  Let $u \in C^{2}\left(\mathbb{R}^{n} \backslash \overline{B_{1}}\right)$ be a radially symmetric  solution of \eqref{equ-equ-2} with $C_0\not=0$, then $u$ is given by  \eqref{Equ-def-u-2} where  $c_0\in\mathbb{R}$ and  $c\in[-(n-1)^{n-1},+\infty)$ or
$[-(n-1)^{n-1},0)$ or $(-\infty,0)$ respectively  for $p\in\{1,2,3\}$. Moreover, if $p=1$, then $u$ is analytic at infinity with expansion \eqref{Equ-asy-radial-2}.

Let $u \in C^{2}\left(\mathbb{R}^{n} \backslash \overline{B_{1}}\right)$ be a radially symmetric  solution of \eqref{equ-equ-2} with $C_0=0$, then $u$ is given by \eqref{Equ-asy-radial-4}, where $c_0\in\mathbb R$ and $c\in\mathbb R\setminus\{0\}$.
\end{theorem}

Under condition \eqref{Case-inverse},
$$
\dfrac{1}{1+\lambda_i}>0\quad\forall~i=1,2,\cdots,n,
$$
hence \eqref{equ-equ-2} implies $C_0<0$. Furthermore, by \eqref{Equ-eigenvlaues}, $W>0$ and $rW'+W>0$ for all $r>1$. Then \eqref{equ-equ-2} implies
$\frac{n-1}{W}+\frac{1}{rW'+W}=C'$ and hence $W>(n-1)C'$. Thus in this case,  $p=1$ in  \eqref{Equ-def-u-2} and $u$ is analytic at infinity.

\subsection{$\tau\in (\frac{\pi}{4},\frac{\pi}{2})$ Case}

When $\tau\in (\frac{\pi}{4},\frac{\pi}{2})$,  equation \eqref{Equ-exterior-domain} reads
\begin{equation}\label{equ-equ-3}
 \frac{\sqrt{a^{2}+1}}{b} \sum_{i=1}^{n} \arctan \frac{\lambda_{i}+a-b}{\lambda_{i}+a+b}=C_0,\quad|x|>1.
\end{equation}
By \eqref{Equ-eigenvlaues}, \eqref{equ-equ-3} becomes
  \begin{equation*}\arctan \frac{U^{\prime \prime}(r)+a-b}{U^{\prime \prime}(r)+a+b}+(n-1) \arctan \frac{\frac{U^{\prime}(r)}{r}+a-b}{\frac{U^{\prime}(r)}{r}+a+b}=C'\quad\text{in}~r>1,
  \end{equation*}
  where $C':=\frac{b}{\sqrt{a^{2}+1}} C_{0}\in (-\frac{n}{2}\pi,\frac{n}{2}\pi)$.
  Let
  \begin{equation*}
  W(r):=\frac{\frac{U'(r)}{r}+a}{b}.
  \end{equation*}
  In order to make \eqref{equ-equ-3} well-defined, $W(r)\not=-1$ for all $r>1$.
  By a direct computation,
  \begin{equation}\label{equ-welldefine-3}
  \arctan \dfrac{W+rW'-1}{W+rW'+1}+(n-1)\arctan \dfrac{W-1}{W+1}=C',
  \end{equation}
  i.e.,
  \begin{equation}\label{equ-temp-4}
  W+rW'=\dfrac{1+\tan\Theta(W)}{1-\tan\Theta(W)}=\tan (\frac{\pi}{4}+\Theta(W))\quad\text{in}~r>1,
  \end{equation}
  where
  $$
  \Theta(w):=C'-(n-1) \arctan \frac{w-1}{w+1}\in (-\frac{\pi}{2},\frac{\pi}{2})\setminus\{\frac{\pi}{4}\}.
  $$
  Since \eqref{equ-temp-4} is a separable differential equation, by a direct computation, it leads to
  \begin{equation}\label{Equ-SPL-radial-1}
  (W^2+1)^{\frac{n-1}{2}}(W\cos(\frac{\pi}{4}+\Theta(W))-
  \sin(\frac{\pi}{4}+\Theta(W)))=cr^{-n}
  \end{equation}
  for some constant $c$, for all $r>1$.

  If $c=0$, then \eqref{Equ-SPL-radial-1} admits a constant solution $W(r)\equiv \tan \left(\frac{\pi}{4}+\frac{C'}{n}\right)$, which implies  quadratic solutions
  $u(x)=-\frac{1}{2}\left(a-b \tan \left(\frac{\pi}{4}+\frac{C'}{n}\right)\right)|x|^{2}+c_0$ for all $c_0\in\mathbb R$.

  If $c\not=0$,
  since \eqref{Equ-SPL-radial-1} holds for all $r>1$, we consider the inverse function $w(\xi)$ of $\xi=G(w)$  on entire $(0,c)$ or $(c,0)$, where
  \begin{equation}\label{equ-sol}
  \left\{
  \begin{array}{llll}
    G(w):=(w^2+1)^{\frac{n-1}{2}}(w\cos(\frac{\pi}{4}+\Theta(w))-\sin(\frac{\pi}{4}+\Theta(w))),\\
    \Theta(w)\in (-\frac{\pi}{2},\frac{\pi}{2})\setminus\{\frac{\pi}{4}\},\quad w\not=-1.\\
  \end{array}
  \right.
  \end{equation}
  By a direct computation,
  \begin{equation}\label{equ-deri-2}
  G'(w)=n(w^2+1)^{\frac{n-1}{2}}\cos(\frac{\pi}{4}+\Theta(w)).
  \end{equation}

  Firstly, we consider the case that $G(w)=0$ is solvable, i.e.,
  $$
  \tan (\frac{\pi}{4}+\Theta(w))=w.
  $$
See for instance the following two graphs of $\xi=\arctan w-\frac{\pi}{4}$, $\xi=\arctan w+\frac{3}{4}\pi$ with $\xi=C'-(n-1)\arctan \frac{w-1}{w+1}$ has a unique intersection in the range of $\xi\in(-\frac{\pi}{2},\frac{\pi}{2})$.
\begin{figure}[htbp]
  \centering
  \includegraphics[width=0.45\linewidth]{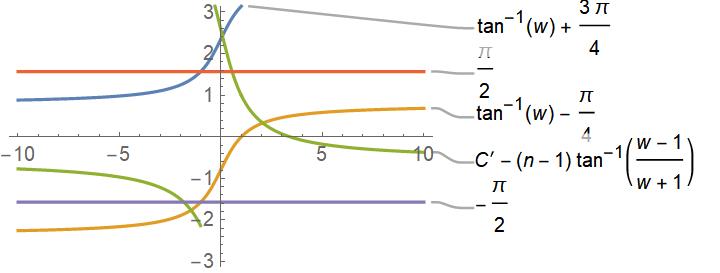}
  \includegraphics[width=0.45\linewidth]{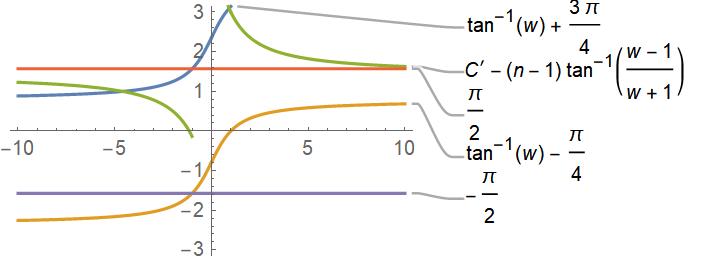}
\end{figure}

For the case of $\arctan w-\frac{\pi}{4}=\Theta(w)$, we have $\arctan w\in (-\frac{\pi}{4},\frac{\pi}{2})$ and thus $w>-1$. By
identity
\begin{equation}\label{equ-identity-1}
\arctan \dfrac{w-1}{w+1}=\arctan w-\frac{\pi}{4}\quad\forall~w>-1,
\end{equation}
(see for instance \cite{huang2019entire,Warren}) the equation becomes
$$
\arctan w-\frac{\pi}{4}=C'-(n-1)\arctan w+\frac{n-1}{4}\pi,
$$
which has a root
$$
w_1(0):=\tan\left(\frac{\pi}{4}+\frac{C'}{n}\right)>-1
$$
if
$-\frac{n}{2}\pi<C'<\frac{n}{4}\pi$. By \eqref{equ-deri-2},
$G(w)$ is monotone increasing as long as
$$
-\frac{\pi}{2}<\frac{\pi}{4}+\Theta(w)<\frac{\pi}{2},
$$
which include the following connected neighbourhood of $w_1(0)>-1$,
$$
\left\{
\begin{array}{llll}
  \left(-1,\tan\left(\frac{n+1}{4n-4}\pi+\frac{C'}{n-1}\right)\right), &
  \text{if }-\frac{n}{2}\pi<C'<-\frac{n}{2}\pi+\frac{3}{4}\pi,
  \\
  \left(
  \tan\left(\frac{n-2}{4n-4}\pi+\frac{C'}{n-1}\right),
  \tan\left(\frac{n+1}{4n-4}\pi+\frac{C'}{n-1}\right)
  \right), & \text{if }
  -\frac{n}{2}\pi+\frac{3}{4}\pi\leq C'\leq \frac{n-1}{4} \pi,
  \\
  \left(
   \tan\left(\frac{n-2}{4n-4}\pi+\frac{C'}{n-1}\right),+\infty
  \right), & \text{if }\frac{n-1}{4} \pi<C'<\frac{n}{4}\pi.\\
\end{array}
\right.
$$

For the case of $\arctan w+\frac{3}{4}\pi=\Theta(w)$,
we have $\arctan w\in (-\frac{\pi}{2},-\frac{\pi}{4})$ and thus $w<-1$. By identity
\begin{equation}\label{equ-identity-2}
\arctan\dfrac{w-1}{w+1}=\arctan w+\frac{3}{4}\pi\quad\forall~ w<-1,
\end{equation}
the equation becomes
$$
\arctan w+\frac{3}{4}\pi=C'-(n-1)\arctan w-\frac{3(n-1)}{4}\pi,
$$
which has a root
$$
w_1(0)=\tan \left(  \frac{1}{4}\pi+\frac{C'}{n}\right)<-1
$$
if $\frac{n}{4}\pi<C'<\frac{n}{2}\pi$. By \eqref{equ-deri-2}, $G(w)$ is monotone decreasing  as long as
$$
\frac{\pi}{2}<\frac{\pi}{4}+\Theta(w)<\frac{3}{4}\pi,
$$
which include the following connected neighbourhood of $w_1(0)<-1$,
$$
\left\{
\begin{array}{llll}
  \left(-\infty,\tan\left(-\frac{3n-2}{4n-4}\pi+\frac{C'}{n-1}\right)\right), &
  \text{if }\frac{n}{4}\pi<C'<\frac{n+1}{4}\pi,\\
  \left(
  \tan\left(-\frac{3n-1}{4n-4}\pi+\frac{C'}{n-1}\right),
  \tan\left(-\frac{3n-2}{4n-4}\pi+\frac{C'}{n-1}\right)
  \right), & \text{if }
  \frac{n+1}{4}\pi\leq C'\leq \frac{n}{2}\pi-\frac{1}{4}\pi
  ,\\
  \left(\tan\left(-\frac{3n-1}{4n-4}\pi+\frac{C'}{n-1}\right),-1
  \right), &
  \text{if }\frac{n}{2}\pi-\frac{1}{4}\pi<C'<\frac{n}{2}\pi.
  \\
\end{array}
\right.
$$

In these  cases, $\xi=G(w)$ admits a unique inverse function $w_1(\xi)$ in  $(\Xi_1,\Xi_2)$, where
$$
0>\Xi_1:=\left\{
\begin{array}{llllll}
-2^{\frac{n}{2}}\sin\left(\frac{n}{2}\pi+C'\right), & -\frac{n}{2} \pi<C'<-\frac{n}{2} \pi+\frac{3}{4} \pi,\\
-\left|\sec \left(\frac{n-2}{4(n-1)}\pi+\frac{C'}{n-1}\right)\right|^{n-1}, & -\frac{n}{2} \pi+\frac{3}{4} \pi \leq C'<\frac{n}{4}\pi,\\
-\left|\sec\left(-\frac{3n-2}{4(n-1)}\pi+\frac{C'}{n-1}\right)\right|^{n-1}, & \frac{n}{4} \pi<C'<\frac{n}{2} \pi-\frac{1}{4}\pi,\\
-2^{\frac{n}{2}}\sin\left(-\frac{n-2}{2}\pi+C'\right),& \frac{n}{2} \pi-\frac{1}{4} \pi<C'<\frac{n}{2} \pi,
\end{array}
\right.
$$
and
$0<\Xi_2:=$
$$
\left\{
\begin{array}{llll}
  \frac{\sqrt{2}}{2}\left|\sec\left(
  \frac{n+1}{4(n-1)} \pi+\frac{C'}{n-1}
  \right)\right|^{n-1}\left(\tan\left(
  \frac{n+1}{4(n-1)} \pi+\frac{C'}{n-1}
  \right)+1\right), & -\frac{n}{2}\pi<C'<\frac{n-1}{4} \pi,\\
  +\infty, & \frac{n-1}{4}\pi\leq C'< \frac{n}{4}\pi,\\
  +\infty, & \frac{n}{4}\pi<C'\leq \frac{n+1}{4}\pi,\\
  -\frac{\sqrt{2}}{2}\left|\sec\left(
  -\frac{3n-1}{4(n-1)}\pi+\frac{C'}{n-1}\right)
  \right|^{n-1}\left(
  \tan\left(
  -\frac{3n-1}{4(n-1)}\pi+\frac{C'}{n-1}
  \right)+1
  \right), & \frac{n+1}{4} \pi< C'<\frac{n}{2}\pi.\\
\end{array}
\right.
$$

Secondly, we consider the case where $G(w)$ converge to $0$ as $w$ goes to infinity or the singular point $w=-1$, which is similar to the $p=2,3$ cases in \eqref{Case-inverse} where $G(w)=\xi$ admits an inverse function on half-neighbourhood of $\xi=0$.

As $w\rightarrow\pm\infty$,
$$
\frac{\pi}{4}+\Theta(w)\rightarrow C'-\frac{n}{4}\pi+\frac{\pi}{2}.
$$
Since $\cos(C'-\frac{n}{4}\pi+\frac{\pi}{2})$ and $\sin(C'-\frac{n}{4}\pi+\frac{\pi}{2})$ are bounded and cannot be zero at the same time, $G(w)$ cannot converge to $0$ as $w$ goes to infinity.

As $w\rightarrow -1$, we have
$$
\lim_{w\rightarrow -1^+}\frac{\pi}{4}+\Theta(w)=C'+\frac{n}{2}\pi-\frac{\pi}{4}\quad\text{and}\quad
\lim_{w\rightarrow -1^-}\frac{\pi}{4}+\Theta(w)=C'-\frac{n}{2}\pi+\frac{3\pi}{4}.
$$
Thus
$$
\lim_{w\rightarrow -1^+}G(w)=2^{\frac{n-1}{2}}(-\cos(C'+\frac{n}{2}\pi-\frac{\pi}{4})
-\sin(C'+\frac{n}{2}\pi-\frac{\pi}{4})),
$$
$$\lim_{w\rightarrow -1^-}G(w)=2^{\frac{n-1}{2}}(-\cos(C'-\frac{n}{2}\pi+\frac{3\pi}{4})
-\sin(C'-\frac{n}{2}\pi+\frac{3\pi}{4})),
$$
and
the only possible cases are $C'=-\frac{n-2}{2}\pi$ with $w>w_2(0):=-1$ and
$C'=\frac{n-2}{2}\pi$ with $w<w_2(0)=-1$ such that the limit becomes zero respectively.
For the first case, by \eqref{equ-deri-2}, $G(w)$ is monotone decreasing in
\begin{equation*}
\left(-1, \tan \left(-\frac{n-2}{4(n-1)} \pi\right)\right).
\end{equation*}
For the second case, by \eqref{equ-deri-2}, $G(w)$ is monotone increasing in
\begin{equation*}
\left\{\begin{array}{ll}
(-\infty,-1), & n=3,4 \\
\left(\tan \left(-\frac{n+2}{4(n-1)} \pi\right),-1\right), & n \geq 5.
\end{array}\right.
\end{equation*}
These imply inverse functions $w_2(\xi)$ of $\xi=G(w)$ in $(0,\Xi_3)$ or $(\Xi_4,0)$ when $C'=\mp\frac{n-2}{2}\pi$ respectively, where
$$
\Xi_3:=\left|\sec \left(\frac{n-2}{4(n-1)} \pi\right)\right|^{n-1},\quad\text{and}\quad
\Xi_4:=\left\{
\begin{array}{lllll}
-\infty, & n=3,4,\\
-\left|\sec \left(\frac{n+2}{4(n-1)} \pi\right)\right|^{n-1}, & n\geq 5.\\
\end{array}
\right.
$$

By \eqref{Equ-SPL-radial-1} and the discussion above,
$$
W(r)=w_p(cr^{-n})\quad\forall~ r>1,
$$
for some $p\in\{1,2\}$ and implies
\begin{equation}\label{Equ-def-u-3}
  u(x)=-\frac{1}{2}(a-b w_{p}(0))|x|^{2}+b\int_{+\infty}^{|x|}(bw_p(c\tau^{-n})-w_p(0)
  )\cdot\tau\mathtt{d} \tau+c_0
\end{equation}
for $c_0\in\mathbb R$.

Furthermore, $G'(w_p(0))\not=0$ and hence  $w_p(\xi)$ are analytic in a neighbourhood of origin.  Thus
\begin{equation}\label{Equ-asy-radial-3}
  u=-\dfrac{1}{2}\left(a-bw_{p}(0)\right)|x|^2+c_0-b|x|^2\sum_{j=1}^{\infty}
  \dfrac{w_p^{(j)}(0)}{(nj-2)j!}(|x|^{-n}c)^j
\end{equation}
for sufficiently large $|x|$ and $p\in \{1,2\}$.

Thus in this case, we have
\begin{theorem}
  Let $u \in C^{2}\left(\mathbb{R}^{n} \backslash \overline{B_{1}}\right)$ be a radially symmetric  solution of \eqref{equ-equ-3}, then $u$ is given by  \eqref{Equ-def-u-3}, where  $c_0\in\mathbb R$, $c\in [\Xi_1,\Xi_2]$ or  $\left(0,\Xi_3\right]$/$\left[\Xi_4,0\right)$ for $p=1,2$ respectively.
Moreover, $u$ is analytic at infinity with expansion \eqref{Equ-asy-radial-3}.
\end{theorem}

Under condition \eqref{Case-large}, either
$
W(r), rW'+W>0
$
or
$
W(r), rW'+W>-1
$ with $|C'+\frac{n\pi}{4}|>\frac{n-2}{2}\pi$. From the proof above, the only possible case is $p=1$.

\subsection{$\tau=\frac{\pi}{2}$ Case}

When $\tau=\frac{\pi}{2}$, equation \eqref{Equ-exterior-domain} reads
\begin{equation}\label{equ-equ-4}
  \sum_{i=1}^{n} \arctan \lambda_{i}=C_0,\quad |x|>1.
\end{equation}
Let $W(r)$ be as in \eqref{equ-W}.   By a direct computation,
\begin{equation}\label{Equ-SPL-welldefined}
  \arctan (W+rW')+(n-1)\arctan W=C_0,
\end{equation}
where $C_0\in (-\frac{n}{2}\pi,\frac{n}{2}\pi)$,
i.e.,
\begin{equation}\label{equ-temp-5}
  W+rW'=\arctan \Theta(W)\quad\text{in}~r>1,
\end{equation}
where
$$
\Theta(w):=C_0-(n-1)\arctan w\in (-\frac{\pi}{2},\frac{\pi}{2}).
$$
Since \eqref{equ-temp-5} is a separable differential equation, by a direct computation, it leads to
\begin{equation}\label{Equ-SPL-radial-2}
  (W^2(r)+1)^{\frac{n-1}{2}}\cdot (W(r)\cos\Theta(W(r))-\sin\Theta(W(r)))=cr^{-n}
\end{equation}
for some constant $c$, for all $r>1$.

If $c=0$, then \eqref{Equ-SPL-radial-2} admits a constant solution $W(r)\equiv w_0:= \tan\frac{C_0}{n}$, which implies quadratic solutions
$u(x)=\frac{1}{2}\tan\frac{C_0}{n}|x|^2+c_0$ for any $c_0\in\mathbb R$.

If $c\not=0$, since \eqref{Equ-SPL-radial-2} holds for all $r>1$, we consider the inverse function $w(\xi)$ of $\xi=G(w)$ on entire $(0,c)$ or $(c,0)$, where
\begin{equation}\label{equ-sol-2}
  \left\{
  \begin{array}{lll}
    G(w):=(w^2+1)^{\frac{n-1}{2}}\cdot (w\cos\Theta(w)-\sin\Theta(w)),\\
    \Theta(w)\in(-\frac{\pi}{2},\frac{\pi}{2}).
  \end{array}
  \right.
\end{equation}
By a direct computation,
\begin{equation}\label{equ-deri-1}
  G'(w)=n(w^2+1)^{\frac{n-1}{2}}\cos(\Theta(w))>0.
\end{equation}
as long as $\Theta(w)\in (-\frac{\pi}{2},\frac{\pi}{2})$, that is
$\arctan w\in (\frac{C_0-\pi/2}{n-1},\frac{C_0+\pi/2}{n-1})$, i.e.,
$$
w\in\left\{
\begin{array}{lll}
  \left(-\infty,\tan\left(\frac{C_0+\pi/2}{n-1}\right)\right), & \text{if }
  -\frac{n}{2}\pi<C_0\leq -\frac{n-2}{2}\pi;\\
  \left(
  \tan\left(\frac{C_0-\pi/2}{n-1}\right),
  \tan\left(\frac{C_0+\pi/2}{n-1}\right)
  \right), & \text{if }-\frac{n-2}{2}\pi<C_0<\frac{n-2}{2}\pi;\\
  \left(
  \tan\left(\frac{C_0-\pi/2}{n-1}\right),+\infty
  \right), & \text{if }\frac{n-2}{2}\pi\leq C_0<\frac{n}{2}\pi.
\end{array}
\right.
$$
Thus $G(w)$ is monotone increasing in the above neighbourhood of $w_0$ and $\xi=G(w)$ admits  a unique inverse function $w(\xi)$ with $w(0)=w_0$ in $(\Xi_1,\Xi_2)$, where
$$
0>\Xi_1:=\left\{
\begin{array}{llll}
  -\infty & -\frac{n}{2}\pi<C_0< -\frac{n-2}{2}\pi,\\
  G(\tan\frac{C_0-\pi/2}{n-1})=-\left|\sec\left(\frac{C_0-\pi/2}{n-1}\right)\right|^{n-1}, & -\frac{n-2}{2}\pi\leq C_0<\frac{n}{2}\pi,\\
\end{array}
\right.
$$
and
$$
0<\Xi_2:=\left\{
\begin{array}{lll}
  G(\tan\frac{C_0+\pi/2}{n-1})=
  \left|\sec\left(\frac{C_0+\pi/2}{n-1}\right)\right|^{n-1}, & -\frac{n}{2}\pi<C_0<\frac{n-2}{2}\pi,\\
  +\infty, & \frac{n-2}{2}\pi\leq C_0<\frac{n}{2}\pi.\\
\end{array}
\right.
$$
By \eqref{Equ-SPL-radial-2} and the discussion above,
$$
W(r)=w(cr^{-n})\quad\forall~r>1
$$
implies
\begin{equation}\label{equ-def-u-5}
u(x)=\frac{1}{2}\tan\frac{C_0}{n}|x|^2+\int_{+\infty}^{|x|}(w(c\tau^{-n})-w(0))\cdot \tau\mathtt{d} \tau+c_0
\end{equation}
for $c_0\in\mathbb R$.

Furthermore, $G'(w_0)>0$ and hence $w(\xi)$ is analytic in a neighbourhood of $\xi=0$. Thus
\begin{equation}\label{Equ-asy-radial-5}
  u(x)=\frac{1}{2}\tan\frac{C_0}{n}|x|^2+c_0-|x|^2\sum_{j=1}^{\infty}
  \dfrac{w^{(j)}(0)}{(nj-2)j!}(c|x|^{-n})^j
\end{equation}
for sufficiently large $|x|$.

Thus in this case, we have
\begin{theorem}
  Let $u\in C^2(\mathbb R^n\setminus\overline{B_1})$ be a radially symmetric solution of \eqref{equ-equ-4}, then $u$ is given by \eqref{equ-def-u-5}, where $c_0\in\mathbb R$ and $c\in [\Xi_1,\Xi_2]$. Moreover, $u$ is analytic at infinity with expansion \eqref{Equ-asy-radial-5}.
\end{theorem}

\section{Asymptotic expansions of general classical solutions}\label{Sec-Expansion-General}

In this section,  we give the asymptotic expansions of linear elliptic equations in subsection \ref{sec-sub-linear}, asymptotic expansions of linearized equation of \eqref{Equ-exterior-domain} and the proof of Theorem \ref{Thm-main-2} in  subsection \ref{Sec-sub-firstStep}.

\subsection{Asymptotic expansions of linear elliptic equations}\label{sec-sub-linear}

In this subsection, we consider the asymptotic expansion at infinity of solution of linear elliptic equation
\begin{equation}\label{Equ-linear}
  a_{ij}(x)D_{ij}v=0\quad\text{in}~\mathbb R^n\setminus\overline B_1,
\end{equation}
where the coefficients are smooth with a positive matrix limit $[a_{ij}(\infty)]>0$ at infinity and $v$ vanishes at infinity. We rewrite \eqref{Equ-linear} into $a_{ij}(\infty)D_{ij}v=(a_{ij}(\infty)-a_{ij}(x))D_{ij}v$ and analyze it by the asymptotic expansion at infinity of Poisson equation
\begin{equation}\label{Equ-Laplacian}
\Delta v=g\quad\text{in}~\mathbb R^n\setminus\overline B_1.
\end{equation}
Consider $g\in C^{\infty}(\mathbb R^n)$ with vanishing speed $g=O(|x|^{-k_1})$ as $|x|\rightarrow+\infty$ for some $k_1>2$. Then
$$
v(x)=\int_{\mathbb R^n}g(y)K(x-y)d y,
$$
is a solution of \eqref{Equ-Laplacian} with vanishing speed $$v=O(|x|^{2-\min\{n,k_1\}})\quad\text{as}~|x|\rightarrow+\infty,
$$
where $\omega_n:=|\mathbb S^{n-1}|$ and $K(x-y):=\frac{1}{(n-2)\omega_n}|x-y|^{2-n}$ is the fundamental solution of Laplace operator, see for instance \cite{Bao-Li-ZhangMA}.
Here we provide the following existence of solution with faster vanishing speed by spherical harmonic expansions as in \cite{Gunther-ConfoNormCoord}.

\begin{lemma}\label{Lem-existence-fastdecay}
  Let $g\in C^{\infty}(\mathbb R^n)$ satisfy
  \begin{equation}\label{Equ-cond-g}
  ||g(r\cdot)||_{L^p(\mathbb S^{n-1})}\leq c_0 r^{-k_1}(\ln r)^{k_2}\quad\forall~r>1
  \end{equation}
  for some $c_0>0,~k_1>2,k_2\ge 0$ and $p>\frac{n}{2}, p\geq 2$. Then there exists a smooth solution
  $v$ of \eqref{Equ-Laplacian}
  such that
  \begin{equation}\label{Equ-exist-1}
  |v(x)|\leq \left\{
  \begin{array}{lllll}
  Cc_0|x|^{2-k_1}(\ln |x|)^{k_2}, & k_1-n\not\in\mathbb{N},\\
  Cc_0|x|^{2-k_1}(\ln |x|)^{k_2+1}, & k_1-n\in\mathbb{N},\\
  \end{array}
  \right.
  \end{equation}
  for some constant $C$ relying only on $n,k_1,k_2$ and $p$.
\end{lemma}

\begin{proof}
Let $\Delta_{\mathbb{S}^{n-1}}$ be the Laplace-Beltrami operator on unit sphere $\mathbb{S}^{n-1}\subset\mathbb{R}^n$ and
$$
\Lambda_0=0,~\Lambda_1=n-1,~\Lambda_2=2n,~\cdots,~\Lambda_k=k(k+n-2),~\cdots,
$$
be the sequence of eigenvalues of $-\Delta_{\mathbb S^{n-1}}$ with eigenfunctions
\begin{equation*}Y_1^{(0)}=1,~Y_{1}^{(1)}(\theta),~Y_{2}^{(1)}(\theta),~\cdots,~ Y_{n}^{(1)}(\theta),~\cdots,~Y_{1}^{(k)}(\theta),~\cdots,~Y_{m_k}^{(k)}(\theta),~\cdots
\end{equation*}
i.e.,
$$
-\Delta_{\mathbb{S}^{n-1}}Y_m^{(k)}(\theta)=\Lambda_kY_m^{(k)}(\theta)\quad\forall~
m=1,2,\cdots,m_k.
$$
The family of eigenfunctions forms a complete standard orthogonal basis of $L^2(\mathbb{S}^{n-1})$.

Expand $g$ and the wanted solution $v$ into
\begin{equation}\label{equ-star}
v(x)=\sum_{k=0}^{+\infty}\sum_{m=1}^{m_{k}} a_{k, m}(r) Y_{m}^{(k)}(\theta)\quad\text{and}\quad
g(x)=\sum_{k=0}^{+\infty}\sum_{m=1}^{m_{k}} b_{k, m}(r) Y_{m}^{(k)}(\theta),
\end{equation}
where
 $r=|x|, \theta=\frac{x}{|x|}$ and $$a_{k,m}(r):=\int_{\mathbb{S}^{n-1}} v(r \theta) \cdot Y_{m}^{(k)}(\theta) \mathtt{d} \theta,\quad b_{k,m}(r):=\int_{\mathbb{S}^{n-1}} g(r\theta) \cdot Y_{m}^{(k)}(\theta) \mathtt{d} \theta.$$
In spherical coordinates,
$$
\Delta v=\partial_{rr}v+\dfrac{n-1}{r}\partial_rv+\dfrac{1}{r^2}\Delta_{\mathbb{S}^{n-1}}v
$$
and \eqref{Equ-Laplacian} becomes
$$
\sum_{k=0}^{\infty} \sum_{m=1}^{m_{k}}\left(a_{ k,m}^{\prime \prime}(r)+\frac{n-1}{r} a_{ k,m}^{\prime}(r)-\frac{\Lambda_{k}}{r^{2}} a_{k,m}(r)\right) Y_{m}^{(k)}(\theta)=
\sum_{k=0}^{+\infty} \sum_{m=1}^{m_{k}} b_{k, m}(r) Y_{m}^{(k)}(\theta).
$$
By the linearly independence of eigenfunctions, for all $k\in\mathbb{N}$ and $m=1,2,\cdots,m_k$,
\begin{equation}\label{Equ-equ-equ}
a_{k,m}^{\prime \prime}(r)+\frac{n-1}{r} a_{ k,m}^{\prime}(r)-\frac{\Lambda_{k}}{r^{2}} a_{k,m}(r) =b_{k,m}(r)\quad\text{in }r>1.
\end{equation}

By solving the ODE, there exist constants $C_{k,m}^{(1)},C_{k,m}^{(2)}$ such that for all $r>1$,
\begin{equation}\label{Equ-def-Wronski}
\begin{array}{lll}
  a_{k,m}(r)&=&C_{k,m}^{(1)}r^k+
C_{k,m}^{(2)}r^{2-n-k}\\
&&\displaystyle
-\dfrac{1}{2-n}r^k\int_{2}^r\tau^{1-k}b_{k,m}(\tau)\mathtt{d} \tau
+\dfrac{1}{2-n}r^{2-k-n}\int_{2}^r\tau^{k+n-1}b_{k,m}(\tau)\mathtt{d} \tau
,
\end{array}
\end{equation}
 By \eqref{Equ-cond-g},
\begin{equation}\label{Equ-converge}
\sum_{k=0}^{+\infty}\sum_{m=1}^{m_k}|b_{k,m}(r)|^2=||g(r\cdot)||^2_{L^2(\mathbb{S}^{n-1})}
\leq c_0^2\omega_n^{\frac{p-2}{p}}r^{-2k_1}(\ln r)^{2k_2}
\end{equation}
for all $r>1$. Then $r^{1-k}b_{k,m}(r)\in L^1(2,+\infty)$  for all $k\in\mathbb N$ and $r^{k+n-1}b_{k,m}(r)\in L^1(2,+\infty)$  for all
$0\leq k<k_1-n, k\in\mathbb N$. We choose $C_{k,m}^{(1)}$ and $C_{k,m}^{(2)}$ in \eqref{Equ-def-Wronski} such that
\begin{equation}\label{Equ-def-v-2}
a_{k,m} (r):=
- \dfrac{1}{2-n}r^k\int_{+\infty}^r\tau^{1-k}b_{k,m}(\tau)\mathtt{d} \tau
+ \dfrac{1}{2-n}r^{2-k-n}\int_{+\infty}^r\tau^{k+n-1}b_{k,m}(\tau)\mathtt{d} \tau
\end{equation}
for all $0\leq k<k_1-n$ and
\begin{equation}\label{Equ-def-v-3}
a_{k,m} (r):=
- \dfrac{1}{2-n}r^k\int_{+\infty}^r\tau^{1-k}b_{k,m}(\tau)\mathtt{d} \tau
+ \dfrac{1}{2-n}r^{2-k-n}\int_{2}^r\tau^{k+n-1}b_{k,m}(\tau)\mathtt{d} \tau
\end{equation}
for all $k\geq k_1-n$

To prove that the series $v(x)$ defined by \eqref{equ-star} converges and obtain its convergence speed, consider
$$
\sum_{k=0}^{+\infty}\sum_{m=1}^{m_k}a_{k,m}^2(r)
=\left\{
\begin{array}{llll}\displaystyle
\sum_{k=0}^{[k_1-n+1]-1}\sum_{m=1}^{m_k}a_{k,m}^2(r)+
\sum_{k=[k_1-n+1]}^{+\infty}\sum_{m=1}^{m_k}a_{k,m}^2(r), & k_1-n\not\in\mathbb N,\\
\displaystyle
\sum_{k=0}^{k_1-n-1}\sum_{m=1}^{m_k}a_{k,m}^2(r)+
\sum_{k=k_1-n+1}^{+\infty}\sum_{m=1}^{m_k}a_{k,m}^2(r)+\sum_{m=1}^{m_{k_1-n}}a_{k_1-n,m}^2(r), & k_1-n\in\mathbb N,\\
\end{array}
\right.
$$
where $[k]$ denotes the largest natural number no larger than $k$.

For $k_1-n\not\in\mathbb N$, we pick $0<\varepsilon:=\frac{1}{2}\min\{1,\mathtt{dist}(k_1-n,\mathbb N)\}$ such that
$$
\left\{
\begin{array}{lll}
3-2k_1+\varepsilon<-1,\\
2n+2k-2k_1-1+\varepsilon<-1 & \forall~ 0\leq k\leq [k_1-n+1]-1,\\ 2n+2k-2k_1-1-\varepsilon>-1 & \forall~ k\geq [k_1-n+1].
\end{array}
\right.
$$
Thus \eqref{Equ-converge} implies
$$
\begin{array}{llll}
  &\displaystyle \sum_{k=0}^{+\infty}\sum_{m=1}^{m_k}a_{k,m}^2(r)\\
  \leq & \displaystyle \dfrac{2}{(n-2)^2}\sum_{k=0}^{+\infty}r^{2k}\left|\int_{+\infty}^{r} \tau^{1-k} b_{k, m}(\tau) \mathrm{d} \tau\right|^2
  +\dfrac{2}{(n-2)^2}\sum_{k=0}^{[k_1-n+1]-1}r^{2(2-k-n)}\left|
  \int_{+\infty}^{r} \tau^{k+n-1} b_{k, m}(\tau) \mathtt{d} \tau
  \right|^2\\&\displaystyle +
  \dfrac{2}{(n-2)^2}\sum_{k=[k_1-n+1]}^{+\infty}r^{2(2-k-n)}\left|
  \int_{2}^{r} \tau^{k+n-1} b_{k, m}(\tau) \mathtt{d} \tau
  \right|^2\\
  \leq &\displaystyle \sum_{k=0}^{+\infty}\sum_{m=1}^{m_{k}} \frac{2r^{2 k}}{(n-2)^{2}} \int^{+\infty}_{r} \tau^{2-2 k} \cdot \tau^{-3-\epsilon} \mathtt{d} \tau  \cdot \int^{+\infty}_{r} \tau^{3+\epsilon} \cdot b_{k, m}^{2}(\tau) \mathtt{d} \tau \\
  &\displaystyle +\sum_{k=0}^{[k_1-n+1]-1}\sum_{m=1}^{m_{k}} \frac{2r^{-2(k+n-2)}}{(n-2)^{2}} \int_{r}^{+\infty} \tau^{2 n+2 k-2 k_{1}-1+\epsilon}(\ln \tau)^{2 k_{2}} \mathtt{d} \tau \cdot \int_{r}^{+\infty} \tau^{2 k_{1}}(\ln \tau)^{-2 k_{2}}b_{k, m}^2(\tau) \frac{\mathtt{d} \tau}{\tau^{1+\epsilon}}\\
  &\displaystyle +\sum_{k=[k_1-n+1]}^{+\infty} \sum_{m=1}^{m_{k}} \frac{2r^{-2(k+n-2)}}{(n-2)^{2}}
  \int_{2}^{r} \tau^{2 n+2 k-2 k_{1}-1-\epsilon}(\ln \tau)^{2 k_{2}} \mathtt{d} \tau\cdot \int_{2}^{r} \tau^{2 k_{1}}(\ln \tau)^{-2 k_{2}} b_{k, m}^2(\tau)\frac{\mathtt{d} \tau}{\tau^{1-\epsilon}} \\
  \leq & \displaystyle \dfrac{2}{(n-2)^2}\sum_{k=0}^{+\infty}\dfrac{r^{-\varepsilon}}{2k+\varepsilon}
  \int_r^{+\infty}\tau^{3+\varepsilon} \sum_{m=1}^{m_k}b_{k,m}^2(\tau)\mathtt{d}\tau\\
  &+\displaystyle
   Cr^{4-2k_1+\epsilon}(\ln r)^{2k_2}
\int^{+\infty}_r
\tau^{2k_{1}}(\ln \tau)^{-2k_{2}}
\sum_{k=0}^{[k_1-n+1]-1}\sum_{m=1}^{m_k}
 b_{k,m}^2(\tau) \dfrac{\mathtt{d} \tau}{\tau^{1+\epsilon}}\\
&+\displaystyle C
  r^{4-2k_1-\epsilon}(\ln r)^{2k_2} \int_2^r
  \tau^{2k_{1}}(\ln \tau)^{-2k_{2}}
\sum_{k=[k_1-n+1]}^{+\infty}\sum_{m=1}^{m_k} b_{k,m}^2(\tau)\dfrac{\mathtt{d} \tau}{\tau^{1-\epsilon}}\\
\leq & Cc_0^2\cdot r^{4-2k_1}(\ln r)^{2k_2}.
\end{array}
$$

For $k_1-n\in\mathbb N$, we pick $\varepsilon:=\frac{1}{2}$. Then
$$
\left\{
\begin{array}{lll}
3-2k_1+\varepsilon<-1,\\
2n+2k-2k_1-1+\varepsilon<-1 & \forall~ 0\leq k\leq k_1-n-1,\\ 2n+2k-2k_1-1-\varepsilon>-1 & \forall~ k\geq k_1-n+1.
\end{array}
\right.
$$
Similar to the calculus above, \eqref{Equ-converge} implies
\begin{equation}\label{equ-estimate-4}
\begin{array}{llll}
  &\displaystyle \sum_{k=0}^{+\infty} \sum_{m=1}^{m_{k}} a_{k, m}^{2}(r)\\
  \leq & \displaystyle Cc_0^2\cdot r^{4-2k_1}(\ln r)^{2k_2}+C
  \sum_{k=k_1-n}\sum_{m=1}^{m_k}r^{2(2-k-n)}\left|\int_2^r\tau^{k+n-1}b_{k,m}(\tau)\mathtt d\tau\right|^2\\
  \leq &  Cc_0^2\cdot r^{4-2k_1}(\ln r)^{2k_2}\\
  & +\displaystyle C\sum_{k=k_1-n}\sum_{m=1}^{m_{k}}
  \int_2^r\tau^{2k_1}(\ln\tau)^{-2k_2}b_{k,m}^2(\tau)\dfrac{\mathtt d\tau}{\tau}\cdot
  \int_2^r\tau^{-1}(\ln\tau)^{2k_2}\mathtt d\tau\\
  \leq &  Cc_0^2\cdot r^{4-2k_1}(\ln r)^{2k_2}\\
  &+\displaystyle Cr^{4-2k_1}(\ln r)^{2k_2+1}\int_2^r\tau^{2k_1}(\ln\tau)^{-k_2}\sum_{m=1}^{m_{k_1-n}}b_{k_1-n,m}^2(\tau)
  \dfrac{\mathtt d\tau}{\tau}\\
  \leq & Cc_0^2\cdot r^{4-2k_1}(\ln r)^{2k_2+2}.\\
\end{array}
\end{equation}
This proves that $v(r)$ is well-defined, is a solution of \eqref{Equ-Laplacian} in distribution sense \cite{Gunther-ConfoNormCoord} and satisfies
\begin{equation}\label{equ-estimate-spherical}
\begin{array}{llllll}
||v(r\cdot)||^2_{L^2(\mathbb S^{n-1})}& \leq& \left\{
\begin{array}{lll}
  C c_0^2 \cdot r^{4-2k_1}(\ln r)^{2k_2}, & k_1-n\not\in\mathbb N,\\
  C c_0^2 \cdot r^{4-2k_1}(\ln r)^{2k_2+2}, & k_1-n\in\mathbb N.
\end{array}
\right.\\
\end{array}
\end{equation}
By interior regularity theory of elliptic differential equations,  $v$ is smooth \cite{GT}. It remains to prove the pointwise decay rate at infinity.

For any $r\gg 1$, we set
$$
v_r(x):=v(rx)\quad\forall~x\in B_4\setminus B_{1}=:D.
$$
Then $v_r$ satisfies
\begin{equation}\label{Equ-scaled}
\Delta v_r=r^2g(rx)=:g_r(x)\quad\text{in}~D.
\end{equation}
By weak Harnack inequality (see for instance Theorem 8.17 of \cite{GT}, see also (2.11) of \cite{Gunther-ConfoNormCoord}),
$$
\sup_{2<|x|<3}|v_r(x)|\leq C(n,p)\cdot \left(
||v_r||_{L^2(D)}+||g_r||_{L^p(D)}
\right).
$$
By \eqref{equ-estimate-spherical},
$$
\begin{array}{llll}
  ||v_r||_{L^2(D)}^2 & = & \displaystyle  \dfrac{1}{r^n}\int_{B_{4r}\setminus B_{r}}|v(x)|^2\mathtt{d} x\\
  &=&\displaystyle r^{-n}\int_{ r}^{4r}||v(\tau\theta)||_{L^2(\mathbb S^{n-1})}^2\cdot \tau^{n-1}\mathtt{d} \tau\\
  &\leq &
  \left\{
  \begin{array}{lll}
  \displaystyle Cc_0^2\cdot r^{-n}\int_{ r}^{4r}
  \tau^{4-2k_1}(\ln\tau)^{2k_2}\cdot \tau^{n-1}\mathtt{d} \tau,
  & k_1-n\not\in\mathbb N,\\
  \displaystyle Cc_0^2\cdot r^{-n}\int_{ r}^{4r}
  \tau^{4-2k_1}(\ln\tau)^{2k_2+2}\cdot \tau^{n-1}\mathtt{d} \tau,
  & k_1-n\in\mathbb N,\\
  \end{array}
  \right.\\
  &\leq & \left\{
  \begin{array}{llll}
    Cc_0^2\cdot  r^{4-2k_1}(\ln r)^{2k_2},  & k_1-n\not\in\mathbb N,\\
    Cc_0^2\cdot  r^{4-2k_1}(\ln r)^{2k_2+2}, & k_1-n\in\mathbb N.\\
  \end{array}
  \right.
\end{array}
$$
By \eqref{Equ-cond-g},
$$
\begin{array}{llll}
  ||g_r||_{L^p(D)}^p & = &\displaystyle
  \dfrac{r^{2p}}{r^n}\int_{B_{4r}\setminus B_{ r}}|g(x)|^p\mathtt{d} x\\
  &\leq & \displaystyle Cc_0^p \cdot r^{2p-n}\int_{ r}^{4r}\tau^{-pk_1}(\ln\tau)^{pk_2}\cdot \tau^{n-1}\mathtt{d} \tau\\
  &\leq & Cc_0^p\cdot r^{2p-pk_1}(\ln r)^{pk_2}.\\
\end{array}
$$
Combining the estimates above, we have
$$
\sup_{2r<|x|<3r}|v(x)|=
\sup_{2<|x|<3}|v_r(x)|\leq
\left\{
\begin{array}{llll}
  Cc_0r^{2-k_1}(\ln r)^{k_2}+Cc_0r^{2-k_1}(\ln r)^{k_2},
  & k_1-n\not\in\mathbb N,\\
  Cc_0r^{2-k_1}(\ln r)^{k_2+1}+Cc_0r^{2-k_1}(\ln r)^{k_2},
  & k_1-n\in\mathbb N,\\
\end{array}
\right.
$$
where $C$ relies only on $n,k_1,k_2$ and $p$.
This finishes the proof of Lemma \ref{Lem-existence-fastdecay}.
\end{proof}

By H\"older inequality, the constant $C$ relying on $p$ in \eqref{Equ-exist-1} remains finite when $p=\infty$ in \eqref{Equ-cond-g}.
For reading simplicity, hereinafter we let $v_g$ denote the solution constructed in Lemma \ref{Lem-existence-fastdecay}.
By Schauder estimates, vanishing speed of derivatives of $v_g$ follow immediately.
\begin{lemma}\label{lem-existence-fastd-deri}
  Let $g\in C^{\infty}(\mathbb R^n)$ satisfy
  \begin{equation}\label{Equ-RHS-converge}
  g=O_{l}(|x|^{-k_1}(\ln |x|)^{k_2})\quad\text{as}~|x|\rightarrow+\infty
  \end{equation}
  for some $k_1>2, k_2\geq 0$, $l-1\in\mathbb N$.   Then
  \begin{equation}\label{Equ-exist-2}
  v_g=\left\{
  \begin{array}{lll}
    O_{l+1}(|x|^{2-k_1}(\ln|x|)^{k_2}), & k_1-n\not\in\mathbb N,\\
    O_{l+1}(|x|^{2-k_1}(\ln|x|)^{k_2+1}), & k_1-n\in\mathbb N.\\
  \end{array}
  \right.
  \end{equation}
\end{lemma}
\begin{proof}
  For sufficiently large $r\gg 1$, let $v_r(x)$ be as in Lemma \ref{Lem-existence-fastdecay}, which satisfies \eqref{Equ-scaled}.
  By a direct computation, for all $0<\alpha<1$, there exists $C>0$ independent of $r$ such that
  $$
  ||g_r||_{C^l(B_4\setminus B_1)}
  \leq Cr^{2-k_1}(\ln r)^{k_2}.
  $$
  By \eqref{Equ-exist-1} in Lemma \ref{Lem-existence-fastdecay}, there exists $C>0$ independent of $r$ such that
  $$
  ||v_r||_{L^{\infty}(B_4\setminus B_{1})}\leq
  \left\{
  \begin{array}{llll}
    Cr^{2-k_1}(\ln r)^{k_2}, & k_1-n\not\in\mathbb N,\\
    Cr^{2-k_1}(\ln r)^{k_2+1}, & k_1-n\in\mathbb N,\\
  \end{array}
  \right.
  $$
  By interior estimates of Schauder type (see \cite{GT}, Chap. 6),
  $$
  \begin{array}{lllll}
  ||v_r||_{C^{l+1,\alpha}(B_{3}\setminus B_{2})}&
  \leq&C\left(
  ||v_r||_{L^{\infty}(B_4\setminus B_{1})}
  +||g_r||_{C^{l-1,\alpha}(B_4\setminus B_1)}
  \right)\\
  &\leq & \left\{
  \begin{array}{llll}
    Cr^{2-k_1}(\ln r)^{k_2}, & k_1-n\not\in\mathbb N,\\
    Cr^{2-k_1}(\ln r)^{k_2+1}, & k_1-n\in\mathbb N,\\
  \end{array}
  \right.
  \end{array}
  $$
  Thus for all $0\leq l_0\leq l+1, l_0\in\mathbb N$,
  $$
  r^{l_0}D^{l_0}v_g(rx)=D^{l_0}v_r(x)\leq  \left\{
  \begin{array}{llll}
    Cr^{2-k_1}(\ln r)^{k_2}, & k_1-n\not\in\mathbb N,\\
    Cr^{2-k_1}(\ln r)^{k_2+1}, & k_1-n\in\mathbb N,\\
  \end{array}
  \right.
  $$
  By the arbitrariness of $r$, this finishes the proof of \eqref{Equ-exist-2}.
\end{proof}

As a consequence, we obtain the following asymptotic expansion of solutions of \eqref{Equ-Laplacian}.
\begin{lemma}\label{Lem-expansion-Laplace}
  Let $g\in C^{\infty}(\mathbb R^n\setminus\overline{B_1})$ satisfy \eqref{Equ-RHS-converge} and $v\in C^2(\mathbb R^n\setminus\overline{B_1})$ be a classical solution of \eqref{Equ-Laplacian} with $v=O(|x|^{2-k_3}(\ln|x|)^{k_4})$ where
  \begin{equation}\label{equ-cond-parameter}
  n\leq k_3<k_1,\quad k_1,k_3,l-1\in\mathbb N\quad\text{and}\quad
  k_2,k_4\geq 0.
  \end{equation}
  Then there exist constants $c_{k,m}$ with $k=k_3-n,\cdots,k_1-n-1$, $m=1,\cdots,m_k$ such that
  \begin{equation}\label{Equ-decompose-Laplace}
  v=
  \sum_{k=k_{3}-n}^{k_{1}-n-1} \sum_{m=1}^{m_{k}} c_{k,m}|x|^{-(k+n-2)} Y_{m}^{(k)}(\theta)+O_{l+1}\left(|x|^{2-k_{1}}(\ln |x|)^{k_{2}+1}\right).
  \end{equation}
\end{lemma}
\begin{proof}
  By Lemma \ref{lem-existence-fastd-deri}, $\widetilde v(x):=v(x)-v_g$ satisfies
  $$
  \Delta \widetilde v=0\quad\text{in}~\mathbb R^n\setminus\overline{B_1}
  $$
  with $$v_g=O_{l+1}(|x|^{2-k_1}(\ln |x|)^{k_2+1})\quad\text{and}
  \quad
  \widetilde v=O(|x|^{2-k_3}(\ln |x|)^{k_4}).
  $$
  Similar to the proof of Lemma \ref{Lem-existence-fastdecay}, we expand $\widetilde v$ into spherical harmonics as
  $$
  \widetilde v(x)=\sum_{k=0}^{+\infty}\sum_{m=1}^{m_k}a_{k,m}(r)Y^{(k)}_m(\theta).
  $$
  It follows from \eqref{Equ-def-Wronski} that there are constants $C_{k,m}^{(1)}, C_{k,m}^{(2)}$ such that  $$
  \widetilde{v}=\sum_{k=0}^{\infty} \sum_{m=1}^{m_{k}} C_{k,m}^{(1)} r^{k} Y_{m}^{(k)}(\theta)+\sum_{k=0}^{\infty} \sum_{m=1}^{m_{k}} C_{k,m}^{(2)} r^{-(k+n-2)} Y_{m}^{(k)}(\theta).
  $$
  By the vanishing speed of $\widetilde v$, we have
  $$
  C_{k,m}^{(1)}=0\quad\forall~k,m;\quad
  C_{k,m}^{(2)}=0\quad\forall~k<k_3-n,~m=1,\cdots,m_k.
  $$
  This finishes the proof by setting $c_{k,m}:=C_{k,m}^{(2)}$.
\end{proof}

By Lemma \ref{Lem-expansion-Laplace} we shall obtain the asymptotic expansions for linear elliptic equation \eqref{Equ-linear}.
For a positive symmetric matrix $[a_{ij}]$, there exists a unique square root matrix $Q:=[a_{ij}]^{\frac{1}{2}}$  such that $Q^T=Q$ and $[a_{ij}]=Q^TQ$.
Also, we let $[a_{ij}]^{-\frac{1}{2}}$ denote the inverse matrix of  $[a_{ij}]^{\frac{1}{2}}$ and $[a_{ij}]^{-1}$ denote the inverse matrix of $[a_{ij}]$.
\begin{proposition}\label{Prop-expansion-linear-perturb}
  Let  $v\in C^2(\mathbb R^n\setminus\overline{B_1})$ be a classical solution of \eqref{Equ-linear} with  $$v=O_2(|x|^{2-k_3}(\ln|x|)^{k_4}),\quad
  a_{i j}(x)-a_{i j}(\infty)=O_{l}\left(|x|^{-k_{5}}(\ln |x|)^{k_{6}}\right)
  $$ where
  $$
  k_3-n,k_5,l-1\in\mathbb N,\quad k_4,k_6\geq 0
  $$
  and $[a_{ij}(\infty)]>0$ being a positive, symmetric matrix.
  Then there exist constants $c_{k,m}$ with $k=k_3-n,\cdots,k_1-n-1$, $m=1,\cdots,m_k$ such that
  $$
  v=
  \sum_{k=k_{3}-n}^{k_3+k_{5}-n-1} \sum_{m=1}^{m_{k}} c_{k,m}(x^T[a_{ij}(\infty)]^{-1}x)^{\frac{2-n-k}{2}} Y_{m}^{(k)}(\theta)+O_{l+1}\left(|x|^{2-k_{3}-k_5}(\ln |x|)^{k_{4}+k_6+1}\right),
  $$
  where
  \begin{equation}\label{Equ-def-Theta}
  \theta=\dfrac{[a_{ij}(\infty)]^{-\frac{1}{2}}x}{(x^T[a_{ij}(\infty)]^{-1}x)^{\frac{1}{2}}}.
  \end{equation}
\end{proposition}
\begin{proof}
  As in Lemma 6.1 of \cite{GT}, let $Q:=[a_{ij}(\infty)]^{\frac{1}{2}}$ and $V(x):=v(Qx)$.
  Since trace is invariant under cyclic permutations,
  $$
  \Delta V(x)=(a_{ij}(\infty)-a_{ij}(Qx))D_{ij}v(Qx)=:g(x)
  $$
  in $Q^{-1}(\mathbb R^n\setminus\overline{B_1})$.
  By a direct computation, $$g=O_{l}(|x|^{-(k_3+k_5)}(\ln|x|)^{k_4+k_6})\quad\text{and}\quad V=O(|x|^{2-k_3}(\ln|x|)^{k_4}).$$ By Lemma \ref{Lem-expansion-Laplace}, there exists constants $c_{k,m}$ with $k=k_{3}-n, \cdots, k_{1}-n-1, m=1, \cdots, m_{k}$ such that $V$ has a decomposition of \eqref{Equ-decompose-Laplace}. The result follows immediately.
\end{proof}

As an application, we state the following special case of Proposition \ref{Prop-expansion-linear-perturb}.
\begin{corollary}\label{Coro-1}
  Let $v$ be a classical solution of \eqref{Equ-linear} with $v=O_l(|x|^{2-n})$ and
  $$
  a_{ij}(x)-a_{ij}(\infty)=O_l(|x|^{-n})\quad\text{as}~|x|\rightarrow+\infty
  $$
  for all $l\in\mathbb N$ with some positive matrix $[a_{ij}(\infty)]$. Then there exist constants $c_{k,m}$ with $k=0,1,\cdots,n-1$, $m=1,\cdots,m_k$ such that
  $$
  v(x)=c_{0,1}(x^T[a_{ij}(\infty)]^{-1}x)^{\frac{2-n}{2}}
  +\sum_{k=1}^{n-1}\sum_{m=1}^{m_k}c_{k,m}(x^T[a_{ij}(\infty)]^{-1}x)^{\frac{2-n-k}{2}}
  Y_{m}^{(k)}(\theta)+O_l(|x|^{2-2n}(\ln|x|))
  $$
  as $|x|\rightarrow+\infty$ for all $l\in\mathbb N$, where  $\theta$ is  as in \eqref{Equ-def-Theta}.
\end{corollary}

\subsection{Asymptotic expansion of general classical solutions}\label{Sec-sub-firstStep}

In this subsection, we analyze the linearized equation of \eqref{Equ-fullyNonlinear-intro} by the asymptotic expansion of linear elliptic equations in subsection \ref{sec-sub-linear} to prove Theorem  \ref{Thm-main-2}.

\begin{lemma}\label{lem-expansion-second}
  Let $u\in C^2(\mathbb R^n\setminus\overline{B_1})$ be a classical solution of
  \eqref{Equ-fullyNonlinear-intro}
  with smooth $F$.  Suppose $u$ satisfies \eqref{ConvergenceSpeed1} for all $l\in\mathbb N$ for some $\gamma\in\mathbb R, \beta\in\mathbb R^n$ and $A\in\mathtt {Sym}(n)$ satisfying $DF(A)>0$ and $F(A)=C_0$. Then there exist constants $c_{k,m}$ with $k=0,1,\cdots,n-1$, $m=1,\cdots,m_k$ such that
  \begin{equation}\label{convergenceSpeed2-ori}
    \begin{array}{lll}
      &\displaystyle u(x)-\left(\frac{1}{2} x^T A x+\beta \cdot x+\gamma\right)\\
      -&\displaystyle \left(
      c_0(x^T(DF(A))^{-1}x)^{\frac{2-n}{2}}
      +\sum_{k=1}^{n-1}c_k(\theta)(x^T(DF(A))^{-1}x)^{\frac{2-n-k}{2}}
      \right)\\
      =& O_l(|x|^{2-2n}(\ln|x|))
    \end{array}
  \end{equation}
  for all $l\in\mathbb N$, where
  \begin{equation}\label{equ-ck-Theta}
  c_k(\theta)=\sum_{m=1}^{m_k}c_{k,m}Y_m^{(k)}(\theta)\in \mathcal H_k^n\quad\text{and}\quad
  \theta=\dfrac{(DF(A))^{-\frac{1}{2}}x}{(x^T(DF(A))^{-1}x)^{\frac{1}{2}}}.
  \end{equation}
\end{lemma}
\begin{proof}
  Since $DF(A)>0$, we may assume that $F$ is uniformly elliptic, otherwise we consider $\widetilde u(x):=R^{-2}u(Rx)$ with sufficiently large $R$ such that $F(D^2\widetilde u)=C_0$ is uniformly elliptic in $\mathbb R^n\setminus\overline{B_1}$. By interior regularity   as in Lemma 17.16 of \cite{GT}, $u$ is smooth in $\mathbb R^n\setminus\overline{B_1}$. Let
  \begin{equation}\label{equ-def-vx}
v(x):=u(x)-\left(\frac{1}{2} x^{T} A x+\beta \cdot x+\gamma\right),\quad x\in \mathbb R^n\setminus\overline{B_1}.
\end{equation}
By a direct computation, $v$ satisfies
\begin{equation*}
a_{i j}(x) D_{i j} v:=\int_{0}^{1} D_{M_{ij}}F\left(t D^{2} v+A\right) \mathtt{d} t \cdot D_{i j} v=0\quad\text {in } \mathbb{R}^{n} \backslash \overline{B}_{1}.
\end{equation*}
By \eqref{ConvergenceSpeed1}, $v=O_l(|x|^{2-n})$ and
$$
|a_{ij}(x)-D_{M_{ij}}F(A)|\leq C|D^2v(x)|=O(|x|^{-n})\quad\text{as}~|x|\rightarrow+\infty.
$$
Expansion \eqref{convergenceSpeed2-ori} follows immediately from Corollary \ref{Coro-1}.
\end{proof}

We finish this section by proving \eqref{ConvergenceSpeed1} under conditions as in Theorem \ref{Thm-main-2}, then Theorem \ref{Thm-main-2} follows from Lemma \ref{lem-expansion-second}. Since the cases in \eqref{Case-MA} and \eqref{Case-SPL}, \eqref{ConvergenceSpeed1} is proved in \cite{CL} and \cite{ExteriorLiouville} respectively, here we only prove for cases \eqref{Case-small}-\eqref{Case-large}.

By interior regularity of fully nonlinear equations, see for instance Lemma 17.16 of \cite{GT}, $u\in C^{\infty}(\mathbb R^n\setminus\overline{B_1})$. By extension theorem as Theorem 6.10 of \cite{EvansMeasureTheory}, we extend $u|_{\mathbb R^n\setminus B_2}$ to $\mathbb R^n$ smoothly such that conditions \eqref{Case-small} and \eqref{Case-inverse} holds on entire $\mathbb R^n$.

\begin{proof}[Proof of Theorem \ref{Thm-main-2}.\eqref{Case-small}]
Let $$
\overline{u}(x):=u(x)+\dfrac{a+b}{2}|x|^2,
$$
then
\begin{equation*}
D^{2} \overline{u}=D^{2} u+(a+b) I>2bI\quad\text{in }\mathbb{R}^n.
\end{equation*}
Let $(\widetilde{x},\widetilde{v})$ be the Legendre transform of $(x,\overline{u})$, i.e.,
$$
\left\{
\begin{array}{ccc}
  \widetilde{x}:=D\overline{u}(x)=Du(x)+(a+b)x,\\
  D_{\widetilde{x}}v(\widetilde{x}):=x\\
\end{array}
\right.
$$
and  we have
\begin{equation*}
D_{\widetilde{x}}^{2} v(\widetilde{x})=\left(D^{2} \overline{u}(x)\right)^{-1}=(D^2u(x)+(a+b)I)^{-1}<\frac{1}{2b}I.
\end{equation*}
Let \begin{equation}\label{LegendreTransform}
\widetilde{u}(\widetilde{x}):=\dfrac{1}{2}|\widetilde{x}|^2-2bv(\widetilde{x}).
\end{equation}
By a direct computation,
\begin{equation}\label{property-Legendre}
\widetilde{\lambda}_{i}\left(D^{2} \widetilde{u}\right)=1-2 b \cdot \frac{1}{\lambda_{i}+a+b}=\frac{\lambda_{i}+a-b}{\lambda_{i}+a+b}\in (0,1).
\end{equation}
Thus $\widetilde{u}(\widetilde{x})$ satisfies the following Monge-Amp\`ere type equation
\begin{equation}\label{temp-16}
\widetilde F(\widetilde{\lambda}(D^2\widetilde{u})):=\sum _ { i = 1 } ^ { n } \ln \widetilde{\lambda_i}=\frac{2b}{\sqrt{a^2+1}}C_0,\ \text{in }\mathbb{R}^n\setminus\overline{\widetilde{\Omega}},
\end{equation}
for some bounded set $\widetilde{\Omega}=D\overline{u}(B_2)=Du(B_2)+(a+b)B_2$.

Moreover,  for any $x\in\mathbb{R}^n,\ \widetilde{x}=D\overline{u}(x)$, $$
|\widetilde{x}-\widetilde{0}|=|Du(x)-Du(0)+(a+b)x|
>2b|x|.
$$
Hence by triangle inequality, \begin{equation}\label{limitofX}
|\widetilde{x}|\geq
-|\widetilde{0}|+|\widetilde{x}-\widetilde{0}|
> -|\widetilde{0}|+2b|x|.
\end{equation}
Especially, \begin{equation}\label{temp-limit}
\lim_{|x|\rightarrow\infty}|\widetilde{x}|=\infty.
\end{equation}

  By \eqref{property-Legendre}, $\ln \widetilde{\lambda_i}(D^2\widetilde{u})\leq 0$ for all $i=1,\cdots,n$. By \eqref{Equ-exterior-domain},
  \begin{equation*}
\ln \widetilde{\lambda}_{i}=\ln \frac{\lambda_{i}+a-b}{\lambda_{i}+a+b} \geq \ln C',
\quad \forall~ i=1,\cdots, n,\quad\text{where}~
C'=\exp\left(\frac{2 b}{\sqrt{a^{2}+1}} C_{0}\right).
\end{equation*}
Thus $\widetilde{\lambda}_{i} \geq C'>0$ and
\begin{equation}\label{temp_lowerbound_1}
  \lambda_i\geq \dfrac{2b}{1-C'}-a-b=-a+b+\dfrac{C'}{1-C'}2b=:-a+b+\delta.
  \end{equation}
By a direct computation, $\frac{\partial \widetilde F}{\partial\widetilde{\lambda_i}}(\widetilde{\lambda})=\frac{1}{\widetilde{\lambda_i}}$ has both positive lower and upper bound, $\frac{\partial^2 \widetilde F}{\partial\widetilde{\lambda_i}^2}(\widetilde{\lambda})=-\frac{1}{\widetilde{\lambda_i}^2}<0$. Hence equation \eqref{temp-16} is uniformly elliptic and concave. By Theorem 2.1 of \cite{ExteriorLiouville},
 there exists $\widetilde{A}\in\mathtt{Sym}(n)$ satisfying
 $$ \widetilde F(\lambda(\widetilde{A}))=\ln C',~\lambda_i(\widetilde{A})\in [C',1],$$ such that
  \begin{equation}\label{LimitofHessian-AfterLegendre}
  \lim_{|\widetilde{x}|\rightarrow+\infty}D^2\widetilde{u}(\widetilde{x})= \widetilde{A}.
  \end{equation}
  By the property of Legendre transform, $$
  D^2\widetilde{u}(\widetilde{x})=I-2b(D^2u(x)+(a+b)I)^{-1},\quad\text{and}\quad
  \widetilde{x}=Du(x)+(a+b)x.
  $$

  Now we prove that all the eigenvalues  $\lambda_i( \widetilde{ A})$  are strictly less than 1, which implies $I-\widetilde{A}$ is an invertible matrix.  By contradiction and rotating the $\widetilde{x}$ -space to make $\widetilde{A}$ diagonal, we suppose that that $\widetilde{A}_{11}=$
  $1.$ Then by the asymptotic behavior of $D\widetilde{u}$ from Theorem 2.1 of \cite{ExteriorLiouville} and hence $Dv$, there exists $\widetilde{\beta}_1\in\mathbb{R}$ such that
  \begin{equation*} D_{\widetilde{x}_{1}}v=\widetilde{\beta}_1+O(|\widetilde{x}|^{1-n})
\end{equation*}
 as $|\widetilde{x}|\rightarrow\infty$. Thus by the definition of Legendre transform (\ref{LegendreTransform}) and (\ref{limitofX}),
\begin{equation}\label{strip-argument}
x_1=D_{\widetilde{x}_1}v(\widetilde{x})=\widetilde{\beta}_{1}+O\left(|\widetilde{x}|^{1-n}\right)
\end{equation}
as $|\widetilde{x}|\rightarrow\infty$.
By (\ref{temp-limit}), (\ref{strip-argument}) also implies that $\mathbb{R}^{n} \backslash \overline{B_2}$ is bounded in the $x_{1}$-direction, hence  a contradiction.  Thus $\lambda_i(\widetilde{A})<1$ strictly for every $i=1,\cdots,n$. Hereinafter we will state similar argument as ``strip argument'' for short, which is also used in Subsection 3.1 of \cite{ExteriorLiouville}.

 By (\ref{limitofX}) and (\ref{LimitofHessian-AfterLegendre}),
 $$
  \lim_{|x|\rightarrow+\infty}D^2u(x)
  =\dfrac{1}{2b}(I-\widetilde{A})^{-1}-(a+b)I,
  $$
  which is a bounded matrix. Together with $u\in C^2(\mathbb{R}^n)$, there exists a constant $M$ such that $$
  D^2u(x)\leq M\quad\forall~x\in\mathbb{R}^n.
  $$
By (\ref{temp_lowerbound_1}),  for $\lambda_i\in [-a+b+\delta,M]$ with $\delta>0$, $$
  \dfrac{\partial F_{\tau}}{\partial \lambda_i}(\lambda)=\frac{\lambda_{i}+a+b}{\lambda_{i}+a-b} \cdot \frac{2 b}{\left(\lambda_{i}+a+b\right)^{2}}=\frac{2 b}{\left(\lambda_{i}+a\right)^{2}-b^{2}}\in \left[\frac{2b}{(M+a)^2-b^2},\frac{2b}{(b+\delta)^2-b^2}\right],
  $$
  and $$
  \dfrac{\partial^2 F_{\tau}}{\partial \lambda_i^2}(\lambda)=-\frac{4 b\left(\lambda_{i}+a\right)}{\left[\left(\lambda_{i}+a\right)^{2}-b^{2}\right]^{2}}<0.
  $$
  Thus $F_{\tau}$ in \eqref{Equ-exterior-domain} is uniformly elliptic and concave under condition \eqref{Case-small} and the result follows from Theorem 2.1 of \cite{ExteriorLiouville}.
\end{proof}

\begin{proof}[Proof of Theorem \ref{Thm-main-2}.\eqref{Case-inverse}] Let $$
\overline{U}(x):=u(x)+\dfrac{1}{2}|x|^2,
$$
then $D^2\overline U>0$ in $\mathbb R^n$.
Let $(\widetilde{x},\widetilde u)$ be the Legendre transform of $(x,\overline{U})$ i.e.
\begin{equation*}
\left\{
\begin{array}{ccc}
  \widetilde{x}:=D\overline{U}(x)=Du(x)+x\\
  D_{\widetilde{x}}\widetilde u(\widetilde{x}):=x\\
\end{array}
\right.,
\end{equation*}
and we have
\begin{equation*}
D_{\widetilde{x}}^2\widetilde u(\widetilde{x})=(D^2\overline{U}(x))^{-1}.
\end{equation*}
Thus $\widetilde u(\widetilde{x})$ satisfies \begin{equation*}
-\Delta_{\widetilde{x}} \widetilde u=\dfrac{\sqrt{2}C_0}{2}\quad\text{in }\mathbb{R}^n\setminus\overline{\widetilde{\Omega}},
\end{equation*}
where $\widetilde{\Omega}=D\overline{U}(B_2)=Du(B_2)+B_2$.
Since $D^2\overline U>0$,
$$
  \dfrac{\sqrt{2}}{\lambda_i(D^2u)+1}
  =\sum_{j\not=i}\dfrac{-\sqrt{2}}{\lambda_j(D^2u)+1}-C_0< -C_0
  $$
   for all $i=1,\cdots,n$.
  Hence \begin{equation*}
  D^2_{\widetilde{x}}\widetilde u(\widetilde{x})
  =(D^2u(x)+I)^{-1}\leq \dfrac{-C_0}{\sqrt{2}}I.
  \end{equation*}
  Thus $D^2_{\widetilde{x}}\widetilde u(\widetilde{x})$ is positive and bounded.
  By Theorem 2.1 of \cite{ExteriorLiouville}, the limit of $D^2_{\widetilde{x}}V(\widetilde{x})$ exists as $|\widetilde{x}|\rightarrow\infty$.

  As in the proof of Theorem \ref{Thm-main-2}.\eqref{Case-small}, the limit of $D^2u(x)$ as $|x|\rightarrow\infty$ also exists by  strip argument (\ref{strip-argument}). Hence $D^2u$ is also bounded from above. Hence $F_{\tau}$ is uniformly elliptic and concave with respect to the sets of solution and the result follows from Theorem 2.1 of \cite{ExteriorLiouville}.
\end{proof}

\begin{proof}[Proof of Theorem \ref{Thm-main-2}.\eqref{Case-large}]

By a direct computation (see for instance \cite{huang2019entire,Warren}), if
 $\lambda_i>-a-b,$ for all $ i=1,2,\cdots,n$, then
  \begin{equation*}
  \sum_{i=1}^n\arctan \frac{\lambda_{i}+a-b}{\lambda_{i}+a+b}=\sum_{i=1}^n\arctan \left(\frac{\lambda_{i}+a}{b}\right)-\frac{n\pi}{4}.
  \end{equation*}
Let
  \begin{equation*}
v(x):=\frac{u(x)}{b}+\frac{a}{2 b}|x|^{2}.
\end{equation*}
 By a direct computation,
  \begin{equation*}
\sum_{i=1}^n\arctan \lambda_i(D^2v)=\frac{b}{\sqrt{a^2+1}}C_0
+\frac{n}{4}\pi
\end{equation*}
in $\mathbb R^n\setminus\overline{B_2}$. When \eqref{condition-temp-2} holds, $D^2v$  satisfies \eqref{equ-cond-spl}. When \eqref{condition-temp-3} holds, we have $\left|\frac{b C_{0}}{\sqrt{a^{2}+1}}+\frac{n \pi}{4}\right|>\frac{n-2}{2} \pi$. By Theorems 1.1 and 1.2 of \cite{ExteriorLiouville}, there exists $\overline {\gamma}\in\mathbb R, \overline {\beta}\in\mathbb R^n$, $\overline A\in\mathtt{Sym}(n)$ and $\overline A\geq 0$  or $\overline A>-\infty$ respectively such that \eqref{ConvergenceSpeed1} holds for $v$. The result follows immediately by the definition of $v$.
\end{proof}

\small

\bibliographystyle{plain}

\bibliography{AsymExpan}

\bigskip

\noindent Z.Liu \& J. Bao

\medskip

\noindent  School of Mathematical Sciences, Beijing Normal University\\
Laboratory of Mathematics and Complex Systems, Ministry of Education\\
Beijing 100875, China \\[1mm]
Email: \textsf{liuzixiao@mail.bnu.edu.cn}\\[1mm]
Email: \textsf{jgbao@bnu.edu.cn}

\end{document}